\pdfoutput=1
\documentclass[english]{interact} 
\usepackage[utf8]{inputenc}
\usepackage{mathtools}
\usepackage{bbm}
\usepackage[numbers,square]{natbib}
\usepackage{comment}
\usepackage{hyperref}
\usepackage{xcolor}

\def\orcid#1{\kern .08em\href{https://orcid.org/#1}{\includegraphics[keepaspectratio,width=0.7em]{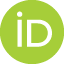}}}  

\usepackage[capitalise,nameinlink]{cleveref}
\numberwithin{equation}{section}
\usepackage{enumitem}

\newtheorem{theorem}{Theorem}[section]
\newtheorem{example}{Example}[section]
\newtheorem{lemma}{Lemma}[section]
\newtheorem{proposition}{Proposition}[section]
\newtheorem{corollary}{Corollary}[section]
\newtheorem{remark}{Remark}[section]
\newtheorem{definition}{Definition}[section]
\newtheorem{assumption}{Assumption}[section]

\newcommand{\R}{\mathbb{R}}
\newcommand{\N}{\mathbb{N}}
\newcommand{\1}{\mathbbm{1}}
\newcommand{\Linop}{\mathbb{L}}

\DeclareMathOperator{\meas}{\mathrm{meas}}
\newcommand{\dx}{\,\mathrm{d}x}
\newcommand{\ds}{\,\mathrm{d}s}
\newcommand{\barK}{K}

\newcommand{\norm}[1]{\|#1\|}
\def\weakto{\rightharpoonup}
\renewcommand{\epsilon}{\varepsilon}

\usepackage[textsize=footnotesize]{todonotes}

\begin{document}

\title{On the no-gap second-order optimality conditions for a non-smooth semilinear elliptic optimal control}
\author{
	\name{Vu Huu Nhu \orcid{0000-0003-4279-3937}\textsuperscript{a}\thanks{CONTACT Vu Huu Nhu. Email: nhu.vuhuu@phenikaa-uni.edu.vn}}
	\affil{\textsuperscript{a} Faculty of Fundamental Sciences, PHENIKAA University, Yen Nghia, Ha Dong, Hanoi 12116, Vietnam}
}

\maketitle

\begin{abstract}
    This work is concerned with second-order necessary and sufficient optimality conditions for optimal control of a non-smooth semilinear elliptic partial differential equation, where the nonlinearity is the non-smooth max-function and thus the associated control-to-state operator is in general not G\^{a}teaux-differentiable. In addition to standing assumptions, two main hypotheses are imposed. The first one is the G\^{a}teaux-differentiability at the considered control of the objective functional and it is precisely characterized by the vanishing of an adjoint state on the set of all zeros of the corresponding state. The second one is a structural assumption on the sets of all points at which the values of the interested state are 'close' to the non-differentiability point of the max-function. We then derive a 'no-gap' theory of second-order optimality conditions in terms of a second-order generalized derivative of the cost functional, i.e., for which the only change between necessary or sufficient second-order optimality conditions is between a strict and non strict inequality.
\end{abstract}


\begin{amscode}
	49K20, 49J20,  35J25, 35J61, 35J91, 90C48
\end{amscode}

\section{Introduction}
In this paper, we consider the following non-smooth semilinear elliptic optimal control problem
\begin{equation}
    \label{eq:P}
    \tag{P}
    \left\{
        \begin{aligned}
            \min_{u\in L^2(\Omega)} &j(u) := \int_\Omega L(x,y_u) \dx + \frac{\nu}{2} \norm{u}_{L^2(\Omega)}^2 \\
            \text{s.t.} \quad &-\Delta y_u + \max(0,y_u) = u \, \text{ in } \Omega, \quad y_u = 0 \, \text{ on } \partial\Omega, \\
            & \alpha(x) \leq u(x) \leq \beta(x) \quad \text{a.e. } x \in \Omega,
        \end{aligned}
    \right.
\end{equation}
where $\Omega$ is a bounded domain in $\R^N$, $N \in \{2,3\}$, with a Lipschitz boundary, $L: \Omega \times \R \to \R$ is a Carath\'{e}odory function and is of class $C^2$ with respect to the second variable, extended measurable functions $\alpha, \beta: \Omega \to [-\infty,\infty]$ satisfy $\beta(x) - \alpha(x) \geq \gamma$ for some $\gamma>0$ and almost everywhere (a.e.) $x \in \Omega$, and $\nu$ is a positive constant. For the precise assumptions on the data of \eqref{eq:P}, we refer to \cref{sec:assumption}.

The state equation in \eqref{eq:P} arises, for instance, in models of the deflection of a stretched thin membrane partially covered by water (see \cite{Kikuchi1984}) and analogous equations arise in free boundary problems for a confined plasma; see, e.g., \cite{Temam:1975,Rappaz:1984}. The salient feature of the state equation is, of course, the non-differentiability of the max-function. This leads to the non-differentiability of the corresponding control-to-state operator \cite{Constantin2017}; see also \cref{prop:control-to-state} below. Consequently, standard techniques for deriving the second-order necessary and sufficient optimality conditions that are based on the second-order differentiability of the control-to-state operator are typically inapplicable, making the analytical and numerical treatment challenging.

The optimal control problem \eqref{eq:P} without control constraints of the form $\alpha \leq u \leq \beta$ was investigated in \cite{Constantin2017} to derive the Clarke- (C-), Bouligand- (B-), and strong stationaries, where the strong stationarity is strongest and the C-stationarity is weakest. However, as seen in \cref{sec:1st-OS}, the objective functional is, indeed, always G\^{a}teaux-differentiable in any stationary point  and thus the notions of these three types of stationarities are equivalent; see \cref{cor:G-diff-no-constraint} below.

\bigskip

Let us comment on related work.  While second-order sufficient optimality conditions (SSC) for optimal control problems governed by smooth PDEs have been intensively studied; see, e.g., \cite{CasasMateos2002,CasasDhamo2011-2ndOSs,Casas2008,KrumbiegelNeitzelRosch2010,RoschTroltzsch2006,RoschTroltzsch2003,Troltzsch2010}, the survey \cite{CasasTroltzsch:2015} and the references therein, the works on SSC for optimal control of non-smooth PDEs are rather rare. To the best of our knowledge, there are comparatively few contributions, such as \cite{KunischWachsmuth2012,Livia2019,ChristofWachsmuth2019} and a recent work \cite{ClasonNhuRosch2020}, on this field. The common approach pursued in these papers is to exploit  a second-order Taylor-type expansion of the mapping $u \mapsto j(u)$, where the objective functional $j$ is, in general, not G\^{a}teaux-differentiable due to the non-differentiability of the control-to-state mapping. In order to derive the SSC, the authors in \cite{KunischWachsmuth2012,Livia2019,ChristofWachsmuth2019} employed an additional sign assumption on the Lagrange multipliers in the vicinity of the contact set that ensures a so-called 'safety distance' \cite[Rem.~4.13]{BetzMeyer2015}. In contrast, the authors in \cite{ClasonNhuRosch2020} used an assumption on the finiteness of oscillation of the considered state around the non-smooth points (see definition (5.10) and Example 5.3 in \cite{ClasonNhuRosch2020}).

Regarding second-order necessary optimality conditions (SNC) for optimal control of non-smooth PDEs, the literature  on this topic is significantly rare. As far as we know the only contribution dealing with SNC for non-smooth PDEs was addressed in \cite{ClasonNhuRosch2020}, where the control-to-state operator is shown to be, in fact, Fr\'{e}chet-differentiable although the coefficient of the PDE is non-smooth. To the best of the authors' knowledge, the investigation of SNC for optimal control problems with non-smooth control-to-state operator is an open research topic. 

The aim of this paper inspired by the work of Clason et al. \cite{ClasonNhuRosch2020} is to obtain second-order necessary and sufficient optimality conditions for \eqref{eq:P} and thus a 'no-gap' theory of these conditions. Unlike the situation in  \cite{ClasonNhuRosch2020}, the control-to-state operator considered in this paper is, however, not G\^{a}teaux-differentiable. This makes problem studied in this paper is more challenging. 
To overcome this difficulty, we first introduce a second-order generalized derivative $Q$ (see \cref{def:curvature-func}) of the objective functional $j$, which mainly depends on the non-smooth point of the max-function. Using a second-order Taylor-type expansion then yields the desired aim  (see \cref{thm:2nd-OS-nec,thm:2nd-OS-suf}). In order to do this, besides the standing assumptions (\cref{ass:standing}), the G\^{a}teaux-differentiability of the objective functional $j$ in the considered control and a structural assumption (see \cref{ass:structure-non-diff} below) have to be fulfilled. Firstly, the G\^{a}teaux-differentiability of $j$ in a control $u$ is characterized via the vanishing of an adjoint state on the set of all zeros of the corresponding state $y$, as shown in \cref{thm:G-diff-characterization} below. Interestingly, when considering the optimal control problem \eqref{eq:P} without control constraints, i.e., $\alpha = -\infty$ and $\beta = \infty$, the G\^{a}teaux-differentiability characterization is automatically satisfied under the regularity of the domain $\Omega$, e.g., $\Omega$ either is a convex domain or has a $C^{1,1}$-boundary. Secondly, the structural-type assumption  was used in \cite{WachsmuthWachsmuth2011,WachsmuthWachsmuth2011-CC} to show the a priori convergence of the regularization error estimates for elliptic optimal control problems, in \cite{Wachsmuth2013} to prove the convergence rates with respect to the discretization parameter for a parameter choice rule of Tikhonov regularization of control-constrained optimal control problems, and in \cite{DeckelnichHinze2012} to prove a priori error estimates for the discretized but unregularized problem.

Recently, the authors in \cite{ChristofWachsmuth2018} have been considered the minimization problem $\min_{u \in C}j(u)$ with a second-order differentiable functional $j$ and a closed, nonempty set $C$. It is well known that the so-called 'sigma-term' $\sigma(\cdot;T^2_C(u,h))$, defined as the value of a support functional of an outer second-order tangent set $T^2_C(u,h)$, contributes prominently in the gap between SNC and SSC (see; e.g., \cite{BonnansShapiro2000}). To remove this gap, the term $Q^{u, \varphi}_C(h)$ was introduced in \cite[Def.~3.1]{ChristofWachsmuth2018} that defines a directional curvature functional for the admissible set $C$ and is very close to the sigma-term $-\sigma(-\varphi;T^2_C(u,h))$ (cf. \cite[Lem.~5.4]{ChristofWachsmuth2018}). The term $Q^{u, \varphi}_C(h)$ can therefore be viewed as a replacement for the  sigma-term. In contrast, the term $Q$ defined in \cref{def:curvature-func} is a replacement of the non-existing second-order derivative of $j$ (see \cref{rem:curvature-comparison} below).

Let us finally emphasize that our results and the underlying analysis can be applied to a non-smooth semilinear elliptic optimal control problem with a finitely $PC^2$-nonlinearity (see \cite{ClasonNhuRosch2020,ClasonNhu2019} for the defintion of finitely $PC^2$-functions).
However, in order to keep the discussion as concise and clear as possible and to be able to focus on the main arguments, we restrict the analysis to \eqref{eq:P}.

\bigskip

The organization of the remainder of this paper is as follows. In the next section, we introduce the notation and the standing assumptions, that will be used throughout the whole paper. In \cref{sec:C2S-oper-G-diff}, we present required properties of the state equation and the characterization of the G\^{a}teaux-differentiability of $j$. \cref{sec:1st-OS} is devoted to the existence of minimizers and the first-order optimality conditions. There, we recall the C-stationarity system for \eqref{eq:P} shown in \cite[Cor.~4.5 \& Rem.~4.9]{Constantin2017} and prove a relaxed optimality system when the objective functional is assumed to be G\^{a}teaux-differentiable in the considered control. Finally, the main results of the paper, the no-gap second-order necessary and sufficient condititions, are proved in \cref{sec:2nd-OC}. 

\section{Notation and standing assumptions}\label{sec:assumption}

\paragraph*{Notation.}
For a given point $u\in X$ and $\rho>0$, $B_X(u,\rho)$ and $\overline B_X(u,\rho)$ stand, respectively, for the open and closed balls of radius $\rho$ centered at $u$. The notation $X \hookrightarrow Y$, for Banach spaces $X,Y$, means that $X$ is continuously embedded in $Y$, and $X \Subset Y$ means that $X$ is compactly embedded in $Y$. 
We denote by $c_0^+$ the set of all positive sequences that converge to zero. For any function $y$, symbols $y^+$ and $y^{-}$ stand, respectively, for the positive and negative parts of $y$. For any local Lipschitz continuous function $f$, the symbol $\partial_C f$ denotes for the Clarke generalized gradient of $f$.
For a function $g:\Omega\to \R$ defined on a domain $\Omega\subset \R^N$ and a subset $M \subset \R$, by $\{ g \in M \}$ we denote the set of all points $x \in \Omega$ for which $g(x) \in M$. Analogously, for given functions $g_1,g_2$ and subsets $M_1, M_2 \subset \R$, the symbol $\{ g_1 \in M_1, g_2 \in M_2 \}$ indicates the set of all points such that the values at which of $g_1$ and $g_2$ belonging, respectively, to $M_1$ and $M_2$. 
For any set $\omega \subset \Omega$, the indicator function of $\omega$ is denoted by $\1_{\omega}$, i.e., $\1_\omega(x) = 1$ if $x \in \omega$ and $\1_\omega(x) =0$ otherwise. Finally, $C$ stands for a generic positive constant, which might be different at different places of occurrence. We also write, e.g., $C_\xi$ for a constant which depends only on the parameter $\xi$.

\medskip

Throughout the paper, we need the following standing assumptions.
\begin{assumption}[standing assumptions for the study of \eqref{eq:P}] \label{ass:standing}
	\leavevmode
	\begin{itemize}
		\item \label{ass:domain}
	    $\Omega \subset \R^N$, $N \in \{2,3\}$, is a bounded domain with a Lipschitz boundary $\partial \Omega$. 
	    \item \label{ass:Tikhonov-para}
	    $\nu >0$ is a given Tikhonov parameter.
	    \item \label{ass:control-constraints}
	    $\alpha, \beta: \Omega \to [-\infty,\infty]$ are measurable functions such that either $\alpha = -\infty$ and $\beta = \infty$ for a.e. $x \in \Omega$ or $\alpha, \beta \in L^2(\Omega)$ with $\beta(x) - \alpha(x) \geq \gamma$ for some $\gamma>0$ and for a.e. $x \in \Omega$.
		\item \label{ass:integrand}
	    The function $L: \Omega \times \R \to \R$ is Carath\'{e}odory such that $L(\cdot,0) \in L^1(\Omega)$ and, for a.e. $x \in \Omega$, the mapping $\R \ni y \mapsto L(x,y) \in \R$  is of class $C^2$. Moreover, for any $M>0$, there exist functions $\phi_M \in L^2(\Omega)$ and $\psi_M \in L^1(\Omega)$ such that
	    \begin{align*}
		    &|L'_y(x,y)| \leq \phi_M(x), \quad |L''_{yy}(x,y)| \leq \psi_M(x) \\
		    \intertext{and}
		    & |L''_{yy}(x,y_1) -L''_{yy}(x,y_2) | \leq \psi_M(x)|y_1 - y_2|
	    \end{align*}
	     for a.e. $x \in \Omega$ and for all $y, y_1, y_2 \in \R$ satisfying $|y|, |y_1|, |y_2| \leq M$.
	     
	     Moreover, in the case $\alpha = - \infty$ and $\beta = \infty$ a.e. in $\Omega$, there exist non-negative functions $\phi, \psi \in L^1(\Omega)$ satisfying
	     $L(x,y) \geq - \phi(x) - \psi(x)|y|$ for a.e. $x \in \Omega$ and for all $y \in \R$.
	\end{itemize}
\end{assumption}

A standard for the choice of $L$ is the quadratic function
\begin{equation*}
	L(x,y) = \frac{1}{2}(y- y_d(x))^2
\end{equation*}
with $y_d \in L^2(\Omega)$.

\medskip

From now on, we will denote the Nemytskii operator from $L^\infty(\Omega)$ to $L^\infty(\Omega)$ associated with the max-function by the same symbol. Similarly, $\max'(y;z)$ stands for the directional derivative of mapping $y \mapsto \max(0,y)$ at the point $y$ in the direction $z$, both considered as a scalar function and as the corresponding Nemytskii operator from $L^\infty(\Omega)$ to itself.

For given measurable functions $y,z$ on $\Omega$, we obviously have
\begin{align*}
	\max(0,y) &= \1_{\{ y>0\} } y \\
	\intertext{and}
	\max{'}(y;z)& = \1_{\{y >0\} } z + \1_{\{y=0 \}} \max(0,z)
\end{align*}
for a.e. in $\Omega$.

\section{Control-to-state operator and characterization of G\^{a}teaux-differentiability of the objective functional} \label{sec:C2S-oper-G-diff}

\subsection{Control-to-state operator} \label{sec:control-to-state}
We shall present in this subsection the Hadamard directional differentiability of the control-to-state operator.

Let us recall the state equation
\begin{equation} \label{eq:state}
	\left\{
	\begin{aligned}
		-\Delta y + \max(0,y) & = u && \text{in } \Omega, \\
		y &=0 && \text{on } \partial\Omega
	\end{aligned}
	\right.
\end{equation}
for $u \in L^2(\Omega)$. 
The existence and uniqueness of solutions in $H^1_0(\Omega)$ of \eqref{eq:state} are shown in \cite[Prop.~2.1]{Constantin2017}. The control-to-state operator $S: L^2(\Omega) \ni u \mapsto y \in H^1_0(\Omega)$ associated with \eqref{eq:state} is then well-defined, globally Lipschitz continuous, and directionally differentiable in the sense of Hadamard \cite[Thm.~2.2]{Constantin2017}. Moreover, since the boundary of $\Omega$ is assumed to be Lipschitz, the $C(\overline\Omega)$-regularity of solutions of \eqref{eq:state} is derived; see, e.g., \cite[Thm.~2.2]{CasasTroltzsch2009} and \cite[Thm.~4.7]{Troltzsch2010}. We therefore consider $S$ as a mapping from $L^2(\Omega)$ to $H^1_0(\Omega) \cap C(\overline\Omega)$.   

\begin{proposition} \label{prop:control-to-state}
	The control-to-state mapping $S: L^2(\Omega) \to H^1_0(\Omega) \cap C(\overline\Omega)$, $u \mapsto y$, associated with the state equation \eqref{eq:state} satisfies the following assertions:
	\begin{enumerate}[label=(\roman*)]
		\item \label{it:Lipschitz} $S$ is globally Lipschitz continuous;
		\item \label{it:dir-der} $S$ is Hadamard directionally differentiable at any $u \in L^2(\Omega)$ in any direction $h \in L^2(\Omega)$. Moreover, for any $u,h \in L^2(\Omega)$, a $S'(u;h) \in H^1_0(\Omega) \cap C(\overline\Omega)$ exists and satisfies
		\begin{equation} \label{eq:weak-strong-Hadamard}
			\frac{S(u + t_k h_k) - S(u)}{t_k} \to S'(u;h) \quad \text{strongly in } H^1_0(\Omega) \cap C(\overline\Omega)
		\end{equation}
		for any sequence $\{h_k\} \subset L^2(\Omega)$ such that $h_k \rightharpoonup h$ in $L^2(\Omega)$ and for any $\{t_k\} \subset c_0^+$. Furthermore, $\delta_h := S'(u;h)$ uniquely solves 
		\begin{equation}
			\label{eq:dir-der}
			\left \{
			\begin{aligned}
				-\Delta \delta_h + \max{'}(y_u;\delta_h) & = h && \text{in } \Omega, \\
				\delta_h &=0 && \text{on } \partial\Omega
			\end{aligned}
			\right.
		\end{equation}
		with $y_u := S(u)$;
		\item \label{it:max-pri} $S'(u;h)$ fulfills the maximum principle, i.e., 
		\[
			h \geq 0 \quad \text{a.e. on } \Omega \quad \implies \quad S'(u;h) \geq 0 \quad \text{a.e. on } \Omega;
		\]
		\item \label{it:wlsc} $S'(u;\cdot)$ is weakly-strongly continuous as function from $L^2(\Omega)$ to $H^1_0(\Omega) \cap C(\overline\Omega)$, i.e.,
		\begin{equation*}
			h_n \rightharpoonup h \quad \text{in } L^2(\Omega) \quad \implies \quad S'(u; h_n) \to S'(u;h)  \quad \text{in } H^1_0(\Omega) \cap C(\overline\Omega).
		\end{equation*}
	\end{enumerate} 
\end{proposition}
\begin{proof}
	\emph{Ad (i):}
	Thanks to \cite[Prop.~2.1]{Constantin2017}, $S$ is globally Lipschitz continuous as a function form $L^2(\Omega)$ to $H^1_0(\Omega)$. 
	To prove assertion \ref{it:Lipschitz}, it remains to show that there is a constant $C>0$ satisfying
	\begin{equation}
		\label{eq:Lipschitz}
		\norm{S(u_1) - S(u_2)}_{C(\overline\Omega)} \leq C \norm{u_1 -u_2}_{L^2(\Omega)} \quad \text{for all } u_1, u_2 \in L^2(\Omega).
	\end{equation}		
	To this end, we subtract the state equations corresponding to $u_1$ and $u_2$ and thus obtain
	\begin{equation*}
		\left\{
		\begin{aligned}
 			-\Delta (y_1 - y_2)  & = u_1 - u_2 - \left( \max(0,y_1)  - \max(0,y_2) \right) && \text{in } \Omega, \\
			y_1 - y_2 &=0 && \text{on } \partial\Omega
		\end{aligned}
		\right.
	\end{equation*}
	with $y_1 := S(u_1)$ and $y_2 := S(u_2)$. Since $L^2(\Omega), C(\overline{\Omega}) \hookrightarrow W^{-1,p}(\Omega)$ for some $p > N$, the right-hand side of the above equation belongs to $W^{-1,p}(\Omega)$. 
	Applying the Stampacchia Theorem \cite[Thm.~12.4]{Chipot2009} and employing the continuous embedding $L^2(\Omega) \hookrightarrow W^{-1,p}(\Omega)$ and the global Lipschitz continuity of the max-function, we derive
	\begin{align*}
		\norm{y_1 - y_2}_{L^\infty(\Omega)} & \leq C_{\Omega, N,p} \left( \norm{u_1-u_2}_{W^{-1,p}(\Omega)} + \norm{\max(0,y_1)  - \max(0,y_2)}_{W^{-1,p}(\Omega)} \right)\\
		& \leq  C_{\Omega, N,p} \left( \norm{u_1-u_2}_{L^2(\Omega)} + \norm{\max(0,y_1)  - \max(0,y_2)}_{L^2(\Omega)} \right) \\
		& \leq C_{\Omega, N,p} \left( \norm{u_1-u_2}_{L^2(\Omega)} + \norm{y_1 - y_2}_{L^2(\Omega)} \right) \\
		& \leq C_{\Omega, N,p} \left( \norm{u_1-u_2}_{L^2(\Omega)} + \norm{y_1 - y_2}_{H^1_0(\Omega)} \right),
	\end{align*}
	where we have just used the continuous embedding $H^1_0(\Omega) \hookrightarrow L^2(\Omega)$ to get the last estimate. From this and the global Lipschitz continuity of $S: L^2(\Omega) \to H^1_0(\Omega)$, we derive \eqref{eq:Lipschitz}.
	
	\emph{Ad (ii):} Due to \cite[Thm.~2.2]{Constantin2017}, $S$ is Hadamard directionally differentiable as a mapping from $L^2(\Omega)$ to $H^1_0(\Omega)$ and its directional derivative fulfills \eqref{eq:weak-strong-Hadamard} as well as the $H^1_0(\Omega)$-strong convergence in \eqref{eq:weak-strong-Hadamard} is valid. It suffices to prove the $C(\overline\Omega)$-strong convergence in \eqref{eq:weak-strong-Hadamard}.   To that end, setting $u_k := u + t_kh_k$, $y_k :=S(u_k)$, and $y := S(u)$, and
	then subtracting equations for $y$ and $\delta_h$ from the one for $y_k$ 
	produces
	\begin{equation*}
		\left\{
		\begin{aligned}
		- \Delta \left( \frac{y_k - y}{t_k} - \delta_h \right)   & = h_k - h - \left( \frac{\max(0,y_k)  - \max(0,y)}{t_k} - \max{'}(y;\delta_h)  \right) && \text{in } \Omega, \\
		\frac{y_k - y}{t_k} - \delta_h  &=0 && \text{on } \partial\Omega.
		\end{aligned}
		\right.		
	\end{equation*}
	Since $\frac{y_k -y}{t_k}$ converges to $\delta_h$ in $H^1_0(\Omega)$, there exists a subsequence, denoted in the same way, such  that $\frac{y_k -y}{t_k} \to \delta_h$ a.e. in $\Omega$, which together with \cite[Lem.~3.5]{ClasonNhuRosch} yields
	\begin{equation*}
		\frac{\max(0,y_k)  - \max(0,y)}{t_k} - \max{'}(y;\delta_h) \to 0 \quad \text{a.e. in } \Omega.
	\end{equation*}
	The Lebesgue dominated convergence theorem then implies that $\frac{\max(0,y_k)  - \max(0,y)}{t_k} - \max{'}(y;\delta_h) \to 0$ in $L^s(\Omega)$ for all $s \geq 1$, in particular for $s=2$. Since $h_k \rightharpoonup h$ in $L^2(\Omega)$ and $L^2(\Omega) \Subset W^{-1,p}(\Omega)$ for some $p >N$, $h_k \to h$ strongly in $W^{-1,p}(\Omega)$ and thus 
	\begin{equation*}
		h_k - h - \left( \frac{\max(0,y_k)  - \max(0,y)}{t_k} - \max{'}(y;\delta_h)  \right) \to 0 \quad \text{strongly in } W^{-1,p}(\Omega).
	\end{equation*}
	The Stampacchia Theorem thus yields $\frac{y_k - y}{t_k} - \delta_h \to 0$ in $L^\infty(\Omega)$,
	which together with a subsequence-subsequence argument gives the $C(\overline\Omega)$-strong convergence in \eqref{eq:weak-strong-Hadamard}.
			
	\emph{Ad (iii):} Let $u, h \in L^2(\Omega)$ be arbitrary such that $h \geq 0$ a.e. in $\Omega$. Setting $y:= S(u)$ and $\delta_h := S'(u;h)$, and testing \eqref{eq:dir-der} by $\delta_h^{-}$, we have
	\[
		- \norm{\nabla \delta_h^{-}}_{L^2(\Omega)}^2 - \norm{  \1_{ \{y>0 \} }  \delta_h^{-}}_{L^2(\Omega)}^2 = \int_\Omega h \delta_h^{-}dx \geq 0.
	\] 
	This implies that $\delta_h^{-} =0$ a.e. in $\Omega$, i.e., $\delta_h \geq 0$ a.e. in $\Omega$.
	
	\emph{Ad (iv):} Take $u \in L^2(\Omega)$ and $h_n \rightharpoonup h$ in $L^2(\Omega)$ and set $\delta_n := S'(u; h_n)$, $\delta := S'(u;h)$. Subtracting the equations for $\delta_n$ and $\delta$ yields
	\begin{equation} \label{eq:ws-convergence}
		\left\{
		\begin{aligned}
			-\Delta(\delta_n - \delta) + \left[ \1_{\{ y_u >0\}}  + \1_{\{ y_u =0, \delta >0 \}}  \right](\delta_n - \delta)    & = \xi_n  && \text{in } \Omega, \\
		\delta_n - \delta  &=0 && \text{on } \partial\Omega
		\end{aligned}
		\right.			
	\end{equation}
	with $\xi_n :=  h_n - h + \1_{\{ y_u =0\}}\left( - \1_{\{\delta_n > 0 \}} + \1_{\{\delta >0 \}}  \right)\delta_n$ and $y_u := S(u)$. 
	Moreover, we have for a.e. in $\Omega$ that
	\begin{align*}
		\left( - \1_{\{\delta_n > 0 \}} + \1_{\{\delta >0 \}}  \right) = \left( - \1_{\{\delta_n > 0, \delta \leq  0 \}} + \1_{\{\delta >0, \delta_n \leq 0 \}}  \right)
	\end{align*}
	and thus that
	\begin{align*}
		&\left( - \1_{\{\delta_n > 0 \}} + \1_{\{\delta >0 \}}  \right)\delta_n (\delta_n -\delta)^{+} = - \1_{\{\delta_n > 0, \delta \leq 0 \}}\delta_n (\delta_n -\delta)^{+} \leq 0
	\end{align*}
	for a.e. in $\Omega$.
	Testing \eqref{eq:ws-convergence} by $(\delta_n - \delta)^{+}$ and applying the
	Cauchy--Schwarz inequality thus gives
	\begin{align*}
		\norm{\nabla (\delta_n - \delta)^{+}}_{L^2(\Omega)}^2  &  \leq \norm{h_n-h}_{H^{-1}(\Omega)}\norm{(\delta_n -\delta)^{+}}_{H^1_0(\Omega)},
	\end{align*}
	which, in association with the  Poincar\'{e} inequality, yields
	\begin{equation*}
		\norm{(\delta_n - \delta)^{+}}_{H^1_0(\Omega)} \leq C \norm{h_n-h}_{H^{-1}(\Omega)}
	\end{equation*}
	for some constant $C$ not depending on $n$. Similarly, there holds
	\begin{equation*}
		\norm{(\delta - \delta_n)^{+}}_{H^1_0(\Omega)} \leq C \norm{h-h_n}_{H^{-1}(\Omega)}.
	\end{equation*}
	We then have
	\begin{equation}
		\label{eq:ws-H1}
		\norm{\delta_n - \delta}_{H^1_0(\Omega)} \leq C \norm{h_n-h}_{H^{-1}(\Omega)}.
	\end{equation}
	On the other hand, since $L^2(\Omega) \Subset W^{-1,p}(\Omega)$ for some $p>N$, the right-hand side of \eqref{eq:ws-convergence} belongs to $W^{-1,p}(\Omega)$. The Stampacchia Theorem then implies that
	\begin{equation*}
		\norm{\delta_n - \delta}_{L^\infty(\Omega)} \leq C \left( \norm{h_n -h}_{W^{-1,p}(\Omega)} + \norm{\1_{\{ y_u =0\}}\left( - \1_{\{\delta_n > 0 \}} + \1_{\{\delta >0 \}}  \right)\delta_n}_{W^{-1,p}(\Omega)} \right).
	\end{equation*}
	A simple computation yields
	\begin{align*}
		\left| \left( - \1_{\{\delta_n > 0 \}} + \1_{\{\delta >0 \}}  \right) \delta_n \right| &\leq   \1_{\{\delta_n > 0, \delta \leq  0 \}}|\delta_n | + \1_{\{\delta >0, \delta_n \leq 0 \}}  |\delta_n |\\
		& \leq \left( \1_{\{\delta_n > 0, \delta \leq  0 \}} + \1_{\{\delta >0, \delta_n \leq 0 \}} \right)  |\delta_n -\delta |
	\end{align*}
	and we therefore have
	\begin{equation*}
		\norm{\1_{\{ y_u =0\}}\left( - \1_{\{\delta_n > 0 \}} + \1_{\{\delta >0 \}}  \right)\delta_n}_{W^{-1,p}(\Omega)}  \leq C \norm{\delta_n - \delta}_{W^{-1,p}(\Omega)}  \leq C \norm{\delta_n - \delta}_{L^{2}(\Omega)}. 
	\end{equation*}
	There thus holds
	\begin{equation*}
		\norm{\delta_n - \delta}_{L^\infty(\Omega)} \leq C \left( \norm{h_n -h}_{W^{-1,p}(\Omega)} + \norm{\delta_n - \delta}_{L^{2}(\Omega)} \right).
	\end{equation*}
	Combining this with \eqref{eq:ws-H1} produces
	\begin{equation*} 
		\norm{\delta_n - \delta}_{L^\infty(\Omega)} \leq C \left( \norm{h_n -h}_{W^{-1,p}(\Omega)} + \norm{h_n-h}_{H^{-1}(\Omega)} \right).
	\end{equation*}
	From this and \eqref{eq:ws-H1}, we derive the desired convergence.
\end{proof}

As a result of \cref{prop:control-to-state}, the following corollary provides a precise characterization of points at which $S$ is G\^{a}teaux-differentiable.
\begin{corollary}[{\cite[cf.][Cor.~2.3]{Constantin2017}}] \label{cor:G-diff-char-S}
	Let $u$ be an arbitrary, but fixed point in $L^2(\Omega)$. Then,
	$S: L^2(\Omega) \to H^1_0(\Omega) \cap C(\overline\Omega)$ is G\^{a}teaux-differentiable in  $u$ if and only if
	\begin{equation*}
		\meas\left( \{S(u)  = 0 \} \right) = 0.
	\end{equation*}
\end{corollary}
\begin{proof}
	The proof of this result is similar to the one of \cite[Cor.~2.3]{Constantin2017}.
\end{proof}

\subsection{Characterization of G\^{a}teaux-differentiability of $j$} \label{sec:G-diff}
In this subsection we shall derive the characterization of G\^{a}teaux-differentiability of the objective functional $j$ that plays an important role in establishing the second-order optimality conditions for \eqref{eq:P}.

We start by defining the solution operators of linear elliptic PDEs. 
\begin{definition} \label{def:sol-oper-adjoint}
	For a given $\chi \in L^\infty(\Omega)$ satisfying $\chi \geq 0$ a.e. in $\Omega$, we define the operator $G_\chi \in \Linop(L^2(\Omega), H^1_0(\Omega) \cap C(\overline\Omega) )$ as follows: for any $h \in L^2(\Omega)$, $z:= G_\chi h$ is the unique solution in $H^1_0(\Omega) \cap C(\overline\Omega)$ to the linear elliptic PDE
	\begin{equation} \label{eq:G-oper}
		\left\{
		\begin{aligned}
			-\Delta z  + \chi z & = h && \text{in } \Omega, \\
			z &=0 && \text{on } \partial\Omega.
		\end{aligned}
		\right.
	\end{equation} 
\end{definition}
The following result precisely characterizes the G\^{a}teaux-differentiable points where the objective functional $j$ is G\^{a}teaux-differentiable.
\begin{theorem}[Characterization of G\^{a}teaux-differentiability of $j$] \label{thm:G-diff-characterization}
	Let $u$ be  arbitrary, but fixed in $L^2(\Omega)$ and let $p:= G_{\1_{\{y_u \geq 0\}}}\left(L'_y(\cdot, y_u) \right)$ with $y_u := S(u)$. Then the following assertions are equivalent:
	\begin{enumerate}[label= (\roman*)]
		\item \label{it:G-diff} 
		$j$ is G\^{a}teaux-differentiable in $u$ and
		\begin{equation}
			\label{eq:diff-formula}
			j{'}(u)h = \int_\Omega (p + \nu u)h \dx 
		\end{equation}
		for all $h \in L^2(\Omega)$.
		\item \label{it:G-diff-char}
		$p$ vanishes a.e. on $\{y_u =0 \}$, i.e.,
		\begin{equation}
			\label{eq:G-diff-char}
			\meas\left( \{y_u = 0 \} \cap \{ p \neq 0\} \right) = 0.
		\end{equation}
	\end{enumerate}
	Moreover, if \eqref{eq:G-diff-char} is fulfilled, then there holds 
	\begin{equation} \label{eq:p-identity}
		p = G_\chi \left(L'_y(\cdot, y_u) \right)
	\end{equation}
	for all $\chi \in L^\infty(\Omega)$ with $\chi(x) \in \partial_C\max(0,y_u(x))$ a.e. in $x \in \Omega$.
\end{theorem}
\begin{proof}
	Since $S$ is directionally differentiable, so is $j$. For any $u, h \in L^2(\Omega)$, setting $\delta_h := S{'}(u;h)$ and using a simple computation yields
	\begin{equation*}
		j{'}(u;h) = \int_\Omega \left( L'_y(x,y_u)\delta_h + \nu uh \right) \dx.
	\end{equation*} 
	Testing the equation for $p$ by $\delta_h$ gives
	\begin{equation*}
		\int_\Omega \left(\nabla p  \cdot \nabla \delta_h + \1_{\{y_u \geq 0 \} } p\delta_h  \right) \dx = \int_\Omega L'_y(x,y_u)\delta_h \dx.
	\end{equation*}
	Similarly, testing the equation for $\delta_h$ via $p$ yields
	\begin{equation*}
		\int_\Omega \left(\nabla p  \cdot \nabla \delta_h + \1_{\{y_u > 0 \} }  p\delta_h + \1_{\{y_u = 0 \} }  p\delta_h^+  \right) \dx = \int_\Omega ph \dx.
	\end{equation*}
	Subtracting the two above equations, we have
	\begin{equation*}
		\int_\Omega L'_y(x,y_u)\delta_h \dx = \int_\Omega p\left( h - \1_{\{y_u = 0 \}} \delta_h^{-} \right) \dx,
	\end{equation*}
	which gives 
	\begin{equation}
		\label{eq:dire-der-j}
		j{'}(u;h) = \int_\Omega ( p + \nu u)h \dx + T(h)
	\end{equation}
	with
	\begin{equation}
		\label{eq:Th}
		T(h) := -\int_\Omega \1_{\{y_u = 0 \}} p \delta_h^{-}\dx.
	\end{equation}
	
	We now prove the equivalence of \ref{it:G-diff} and \ref{it:G-diff-char}. 
	
	\emph{$\impliedby)$} Assume that \ref{it:G-diff-char} is fulfilled. Then $T(h) =0$ and thus $j{'}(u;h) = \int_\Omega ( p + \nu u)h \dx$ according to \eqref{eq:dire-der-j}. Consequently, $j{'}(u;h)$ is linear in $h$ and $j$ is therefore G\^{a}teaux-differentiable in $u$. We hence derive \ref{it:G-diff}.
	
	\emph{$\implies )$} Assume that $j$ is G\^{a}teaux-differentiable in $u$. Then $T$ is linear. We first show that
	\begin{equation}
		\label{eq:p-vanish-1}
		\meas\left( \{ y_u = 0\} \cap \{p>0 \} \right) = 0.
	\end{equation}
	For this purpose, it suffices to consider the case $\meas\{p>0\} >0$. Since $L'_y(\cdot, y_u) \in L^2(\Omega)$, $p$ is continuous on $\overline{\Omega}$ because of the regularity of solutions of \eqref{eq:G-oper}. Accordingly, the set $\{p>0\}$ is open and thus there exists a $\varphi \in C^{\infty}(\R^N)$ such that 
	\begin{equation*}
		\varphi|_{\{p>0\}} > 0 \quad \text{and} \quad \varphi|_{\R^N \backslash \{p>0\}} = 0
	\end{equation*}	
	(\cite[see][Lem.~A.1]{Constantin2017}).
	Setting 
	\begin{align*}
		h_1 &:= - \Delta \varphi + \max{'}(y_u,\varphi) \quad
		\text{and} \quad
		h_2 :=  - \Delta (-\varphi) + \max{'}(y_u,-\varphi)
	\end{align*}
	yields $h_1 + h_2 = \1_{\{y_u =0\}} \varphi \geq 0$.
	Putting $\delta := S{'}(u;h_1+h_2)$ and applying the maximum principle (see \cref{prop:control-to-state}), we derive $\delta \geq 0$ a.e. in $\Omega$. Obviously, $T(h_1) = 0 = T(h_1+h_2)$ and $T(h_2) = - \int_\Omega \1_{\{y_u = 0 \}} p \varphi \dx$. The linearity of $T$ then implies that 
	\begin{equation*}
		\int_\Omega \1_{\{y_u = 0 \}} p \varphi \dx = 0,
	\end{equation*}
	which, together with the definition of $\varphi$, shows \eqref{eq:p-vanish-1}.
	
	In the same way, we have 
	\begin{equation*}
		\meas\left( \{ y_u = 0\} \cap \{p <0 \} \right) = 0.
	\end{equation*}
	Combining this with \eqref{eq:p-vanish-1} yields \eqref{eq:G-diff-char}.
	
	It remains to prove \eqref{eq:p-identity}. To do this, take any $\chi \in L^\infty(\Omega)$ such that $\chi(x) \in \partial_C \max(0,y_u(x))$ a.e. $x \in \Omega$. By virtue of \eqref{eq:G-diff-char}, we have
	\begin{equation*}
		\1_{\{y_u \geq 0 \}} p  = \chi p \quad \text{a.e. in } \Omega.
	\end{equation*}
	We then derive \eqref{eq:p-identity} from \cref{def:sol-oper-adjoint}.
\end{proof}
\begin{remark} \label{rem:G-diff-char}
	In the next section we shall see under an additional regularity of $\Omega$ and the assumption $\alpha < 0 < \beta$  that the G\^{a}teaux-differentiability characterization \eqref{eq:G-diff-char} is fulfilled at any stationary control point, i.e., the control that fulfills the optimality conditions for \eqref{eq:P}; see \cref{def:C-stationary} and \cref{prop:G-diff-exams} below. As a result, when considering \eqref{eq:P} without control constraints of the form $\alpha \leq u \leq \beta$, the G\^{a}teaux-differentiability of the objective functional is always achieved at any stationary point and thus the stationarity notions in the Clarke-, Bouligand-, and strong senses are equivalent at these points (see \cref{cor:G-diff-no-constraint} below).
\end{remark}

\section{Existence and first-order optimality conditions} \label{sec:1st-OS}
This section is devoted to presenting the existence of global optimal controls and the first-order optimality conditions for \eqref{eq:P} that shall be exploited, in some certain situations, to verify the G\^{a}teaux-differentiability characterization \eqref{eq:G-diff-char}. 
Conversely, the optimality conditions will be relaxed due to the G\^{a}teaux-differentiability of the objective functional $j$.

We first rewrite the optimal control problem \eqref{eq:P} in the form
\begin{equation}\label{eq:P2}
	\tag{P}
	\left\{
	\begin{aligned}
		\min_{u\in L^2(\Omega)} & j(u)  = \int_\Omega L(x, S(u)) \dx + \frac\nu2\norm{u}_{L^2(\Omega)}^2 \\
		\text{s.t.} &  \quad u \in \mathcal{U}_{ad}, 
	\end{aligned}
	\right.
\end{equation}
where the admissible set is defined by 
\begin{equation*}
	\mathcal{U}_{ad} := 
	\left\{
	\begin{aligned}
		&L^2(\Omega) &&\text{if } \alpha = -\infty, \beta = \infty, \\
		&\left\{u \in L^2(\Omega) \,\middle|\, \alpha(x) \leq u(x) \leq \beta(x) \, \text{for a.e. } x \in \Omega  \right\} && \text{if } \alpha,  \beta \in L^2(\Omega). 
	\end{aligned}
	\right.	
\end{equation*}
\begin{proposition} \label{prop:existence-minimizer}
	The optimal control problem \eqref{eq:P} admits at least one global minimizer $\bar u \in \mathcal{U}_{ad}$.
\end{proposition}
\begin{proof}
	The argument of the proof is standard for the case $\alpha, \beta \in L^2(\Omega)$; see, e.g., \cite[Thm.~3.1]{CasasTroltzsch2009}. We thus only consider the situation that $\alpha = -\infty$ and $\beta = \infty$. Let $\{u_k\}$ be a minimizing sequence for \eqref{eq:P}, i.e.,
	\begin{equation*}
		\lim\limits_{k\to \infty } j(u_k) = \inf(\text{P}).
	\end{equation*}
	In view of \cref{ass:standing}, there exist non-negative functions $\phi, \psi \in L^1(\Omega)$ such that
	\begin{equation*}
		L(x,y_u(x)) \geq - \left( \phi(x) + \psi(x)|y_u(x)| \right)
	\end{equation*}
	for a.e. $x \in \Omega$ and for all $u \in L^2(\Omega)$ with $y_u := S(u)$. Besides, there exists a constant $C>0$ independent of $u$ such that
	\begin{equation*}
		\norm{y_u}_{L^\infty(\Omega)} \leq C \norm{u}_{L^2(\Omega)}.
	\end{equation*}
	We thus have for any $u \in L^2(\Omega)$ that
	\begin{equation*}
		j(u) \geq - \left( \norm{\phi}_{L^1(\Omega)} + C \norm{\psi}_{L^1(\Omega)} \norm{u}_{L^2(\Omega)} \right) + \frac{\nu}{2} \norm{u}_{L^2(\Omega)}^2.
	\end{equation*}
	It then follows that $\inf(\text{P}) > -\infty$ and that
	\begin{equation*}
		\liminf_{k \to \infty}\left[ - \left( \norm{\phi}_{L^1(\Omega)} + C \norm{\psi}_{L^1(\Omega)} \norm{u_k}_{L^2(\Omega)} \right) + \frac{\nu}{2} \norm{u_k}_{L^2(\Omega)}^2\right] \leq   \inf(\text{P}),
	\end{equation*}
	which yields the boundedness of a subsequence of $\{u_k\}$ in $L^2(\Omega)$. A standard argument then completes the proof.
\end{proof}

We now give a definition on the notion of stationarity.
\begin{definition} \label{def:C-stationary}
	A feasible point $\bar u \in \mathcal{U}_{ad}$ is said to be a \emph{C-stationary} point of \eqref{eq:P} if there exist an adjoint state $\bar p \in H^1_0(\Omega) \cap C(\overline\Omega)$ and a multiplier $\bar\chi \in L^\infty(\Omega)$ satisfying 
	\begin{subequations}
		\label{eq:1st-OS}
		\begin{align}
		&\left\{
		\begin{aligned}
		-\Delta \bar p + \bar \chi \bar p &= L'_y(x,\bar y) && \text{in } \Omega, \\
		\bar p &=0 && \text{on } \partial\Omega,
		\end{aligned}
		\right.  \label{eq:adjoint_OS} \\
		& \bar \chi(x) \in \partial_C \max(0, \bar y(x)) \quad \text{a.e. } x  \in \Omega, \label{eq:Clarke-grad}
		\intertext{and}
		&\int_\Omega \left(\bar p + \nu \bar u \right)\left( u - \bar u \right) \dx \geq 0 \quad \text{for all } u \in \mathcal{U}_{ad}  \label{eq:normal_OS}
		\end{align}
	\end{subequations}
	with $\bar y:= S(\bar u)$.
\end{definition}

Exploiting regularization and relaxation methods as in, e.g., \cite{Barbu1984}, we derive the stationarity system for \eqref{eq:P}. The same approach to deal with the optimal control governed by non-smooth PDEs was used in \cite{MeyerSusu2017,Constantin2017,ClasonNhuRosch}. 
\begin{theorem}[{\cite[Cor.~4.5 \& Rem.~4.9]{Constantin2017}}]
	\label{thm:1st-OC}
	If $\bar u$ a local minimizer of \eqref{eq:P2}, then $\bar u$ is a C-stationary point.	
\end{theorem}

\begin{proposition} 	\label{prop:G-diff-exams}
	Assume that $\Omega$ either is convex or has a $C^{1,1}$-boundary and that $\alpha(x) <0 < \beta(x)$ for a.e. $x \in \Omega$. Let $\bar u  \in \mathcal{U}_{ad}$ and $\bar p \in L^2(\Omega)$ be arbitrary and satisfy \eqref{eq:normal_OS}. Then, there holds
	\begin{equation*}
		\meas\left( \{ \bar y = 0 \} \cap \{ \bar p \neq 0 \} \right) = 0
	\end{equation*}
	with $\bar y:= S(\bar u)$.
\end{proposition}
\begin{proof}
	It is sufficient to show that
	\begin{equation} \label{eq:p-vanish-2}
		\bar p = 0 \quad \text{a.e. on } \{\bar y = 0\}.
	\end{equation}
	To do this, we first employ the regularity of the domain $\Omega$ to derive the $H^2(\Omega)$-regularity of  $\bar y$. From this and the behavior of derivatives on level sets; see, e.g., \cite[Lem.~4.1]{ChristofMuller2020}, we have that $\bar u := -\Delta \bar y + \max(0, \bar y)$ always vanishes a.e. on the set $\{\bar y =0 \}$. This, \eqref{eq:normal_OS}, and the assumption that $\alpha(x) <0 < \beta(x)$ for a.e. $x \in \Omega$ imply that $\bar p$ vanishes a.e. in $\{\bar u = 0\}$ and thus \eqref{eq:p-vanish-2}.
\end{proof}

The stationarity system \eqref{eq:1st-OS} is relaxed when the objective functional is G\^{a}teaux-differentiable at the local minimizer.
\begin{theorem}[relaxed optimality system] \label{thm:1st-OS-relaxed}
	Let $\bar u$ be a local minimizer of \eqref{eq:P2} and $\bar y$ be the corresponding state. Assume that $j$ is G\^{a}teaux-differentiable in $\bar u$. Then, there exists an adjoint state $\bar p \in H^1_0(\Omega) \cap C(\overline\Omega)$ satisfying
	\begin{subequations}
		\label{eq:1st-OS-relaxed}
		\begin{align}
			&\left\{
			\begin{aligned}
				-\Delta \bar p +  \chi \bar p &= L'_y(x,\bar y) && \text{in } \Omega, \\
				\bar p &=0 && \text{on } \partial\Omega,
			\end{aligned}
			\right.  \label{eq:adjoint_OS-relaxed} \\
			&\int_\Omega \left(\bar p + \nu \bar u \right)\left( u - \bar u \right) \dx \geq 0 \quad \text{for all } u \in \mathcal{U}_{ad}  \label{eq:normal_OS-relaxed}
		\end{align}
	\end{subequations}
	for all $\chi \in L^\infty(\Omega)$ with $\chi(x) \in \partial_C \max(0, \bar y(x))$ a.e. $x \in \Omega$. Moreover, there holds
	\begin{equation} \label{ass:G-diff}
	 	\tag{GA}
		\meas\left( \{\bar y = 0 \} \cap \{ \bar p \neq 0\} \right) = 0.
	\end{equation}
\end{theorem}
\begin{proof}
	Setting $\bar p :=G_{ \1_{\{\bar y \geq 0 \}}}(L'_y(\cdot, \bar y))$ and applying \cref{thm:G-diff-characterization}, we have for all $h \in L^2(\Omega)$ that
	\begin{equation*}
		j'(\bar u)h = \int_\Omega (\bar p + \nu \bar u)h \dx
	\end{equation*}
	and $\bar p$ satisfies \eqref{ass:G-diff}.
	Let $u$ be an arbitrary point in $\mathcal{U}_{ad}$. Since $\bar u$ is a local optimal solution of \eqref{eq:P}, there holds
	\begin{equation*}
		\frac{j(\bar u + t(u -\bar u)) - j(\bar u)}{t} \geq 0
	\end{equation*}
	for all $t>0$ small enough. Letting $t \to 0^+$ yields $j'(\bar u)(u - \bar u) \geq 0$, which is identical to \eqref{eq:normal_OS-relaxed}. Finally, \eqref{eq:adjoint_OS-relaxed} follows from the combination of the definition of $\bar p$ and \eqref{eq:p-identity}.
\end{proof}

The following result, which is a direct consequence of \cref{prop:G-diff-exams}, \cref{thm:G-diff-characterization}, and \cref{thm:1st-OS-relaxed}, presents the G\^{a}teaux-differentiability of the objective functional at any C-stationary point  and, as stated in \cite[Cor.~5.4]{ChristofMuller2020}, the equivalence of some notions of stationarity when considering \eqref{eq:P} without control constraints.

\begin{corollary}
	\label{cor:G-diff-no-constraint}
	Assume that either $\alpha = -\infty$ and $\beta = \infty$  or $\alpha < 0 < \beta$ a.e. in $\Omega$. Let $\Omega$ either be convex or have a $C^{1,1}$-boundary. Then, the following assertions hold:
	\begin{enumerate}[label=(\roman*)]
		\item The objective functional $j$ is G\^{a}teaux-differentiable in any C-stationary point.
		\item If $\bar u$ is a C-stationary point, then $\bar p :=G_{ \1_{\{\bar y \geq 0 \}}}(L'_y(\cdot, \bar y))$ satisfies \eqref{eq:1st-OS-relaxed}.
		\item If $\alpha = -\infty$ and $\beta = \infty$, i.e., \eqref{eq:P} has no control constraints, then the notions of C-stationarity, purely primal stationarity in the sense of \cite[Prop.~4.10]{Constantin2017}, and strong stationarity in the sense of \cite[Thm.~4.12]{Constantin2017} are identical.
	\end{enumerate}
\end{corollary}

\section{Second-order optimality conditions} \label{sec:2nd-OC}

Our main goal in this paper is to derive second-order necessary and sufficient optimality conditions in terms of a \emph{second-order generalized derivative} of the cost functional $j$ in critical directions. 

Throughout this section, let $\bar u$ be an arbitrary, but fixed feasible control in $\mathcal{U}_{ad}$ and let $\bar y:= S(\bar u)$. Define $\bar p : = G_{\1_{\{ \bar y \geq 0 \}}}\left(L'_y(\cdot, \bar y) \right)$, i.e., $\bar p \in H^1_0(\Omega) \cap C(\overline\Omega)$ is the unique solution to
\begin{equation} \label{eq:adjoint}
	\left\{
	\begin{aligned}
	-\Delta p  + \1_{\{ \bar y \geq 0 \}} \bar p & = L'_y(\cdot, \bar y) && \text{in } \Omega, \\
	\bar p &=0 && \text{on } \partial\Omega,
	\end{aligned}
	\right.
\end{equation}
and set
\begin{equation}
	\label{eq:objective-der}
	\bar d:= \bar p + \nu \bar u.  
\end{equation}

\subsection{Second-order generalized derivatives} \label{sec:2nd-OC:curvature}
Before establishing second-order generalized derivatives of $j$ we need some additional notation. 

For any $t>0$ and $h \in L^2(\Omega)$, we define
\begin{equation} \label{eq:zeta-func}
    \begin{aligned}
        \zeta(\bar u; t, h) &:=   S(\bar u + th)\left( \1_{\{\bar y\geq 0 \}}  - \1_{\{ S(\bar u + th) \geq 0 \}} \right).
    \end{aligned}
\end{equation}
We then define for $\{t_n\} \in c_0^+$, $h \in L^2(\Omega)$, and $\{h_n\} \subset L^2(\Omega)$, the extended numbers
\begin{equation}
    \label{eq:key-term-sn}
    \underline{Q}(\bar u,\bar p;\{t_n\}, \{h_n\}) := \liminf_{n \to \infty} \frac{1}{t_n^2} \int_\Omega \bar p  \zeta (\bar u;t_n, h_n) \dx
\end{equation}
and
\begin{equation}
    \label{eq:key-term}
    \tilde{Q}(\bar u,\bar p; h) := \inf \left\{ \underline{Q}(\bar u,\bar p;\{t_n\}, \{h\}) \mid \{t_n\} \in c_0^+ \right\}
\end{equation}
with $\{h\}$ denoting the constant sequence $h_n \equiv h$ for all $n\geq 1$.

The SNC and SSC for $\bar u$ of \eqref{eq:P} are relevant to the following notion.
\begin{definition}
	\label{def:curvature-func}
	For any $h \in L^2(\Omega)$, we call $Q(\bar u, \bar p; h)$ with
	\begin{equation}
		\label{eq:curvature-func}
		Q(\bar u, \bar p; h) :=  \int_\Omega\left[ L''_{yy}(x,\bar y) \left(S'(\bar u; h)\right)^2 + \nu h^2 \right] \dx + 2 \tilde{Q}(\bar u,\bar p; h)
	\end{equation}
	a \emph{second-order generalized derivative} of $j$ at $\bar u$ in the direction $h$.
\end{definition}

\begin{remark} \label{rem:curvature-comparison}
	As shall be seen in \cref{thm:2nd-OS-nec,thm:2nd-OS-suf} the term  $Q(\bar u, \bar p; h)$ occurs in both (SNC) and (SSC), and it plays an important role in deriving the no-gap conditions. Moreover, from the lines in proofs of \cref{thm:2nd-OS-nec,thm:2nd-OS-suf}, the second-order generalized derivative $\frac{1}{2}Q(\bar u, \bar p; h)$ is related to the limit inferior of
	\[
		\frac{1}{t_n^2}[j(\bar u + t_n h_n) - j(\bar u)]
	\]
	with some sequences $t_n \to 0^+$ and $h_n \rightharpoonup h$ weakly in $L^2(\Omega)$. Therefore it can be regarded as a replacement for the non-existing second-order derivative of $j$. 
%
%
\end{remark}

The term $\tilde{Q}$ plays an important role in the second-order generalized derivative of $j$ and, as seen in \cref{sec:2nd-OC:soc} below, it significantly contributes to deriving the no-gap theory of second-order necessary and sufficient conditions. In the rest of this subsection, we provide some required properties of $\tilde{Q}$. For this purpose, we need the following assumption on the structure of sets near the boundary of $\{\bar y =0\}$.
\begin{assumption} \label{ass:structure-non-diff}
	There exists a positive constant $c_s$ that satisfies
	\begin{equation}
		\label{eq:structure-non-diff}
		\tag{SA}
		\meas\left( \{ 0 < |\bar y| < \epsilon \} \right) \leq c_s \epsilon
	\end{equation}
	for all $\epsilon \in (0,1)$.
\end{assumption}

The assumption \eqref{eq:structure-non-diff} does not require that the set $\{\bar y =0 \}$ has measure zero; see \cref{exam:structure-ass} below. Furthermore, \cref{exam:SA-GA} shows that both assumptions \eqref{ass:G-diff} and \eqref{eq:structure-non-diff} do not imply that meas$\{\bar y =0 \}=0$.
The control-to-state mapping $S$ is thus not G\^{a}teaux-differentiable in $\bar u$ in general due to \cref{cor:G-diff-char-S}. A condition that is the same as \eqref{eq:structure-non-diff} but imposed on an adjoint state was employed in \cite{WachsmuthWachsmuth2011} to derive error
estimates with respect to the regularization parameter for an elliptic optimal control problem. Such a condition was also used in  \cite{Wachsmuth2013}  to deal with the parameter choice rule for the Tikhonov regularization parameter depending on a posteriori computable quantities. 
In \cite{DeckelnichHinze2012}, a more general version of \eqref{eq:structure-non-diff}  was exploited to show the a priori error estimates for the approximation of elliptic control problems.

\begin{example}
	\label{exam:structure-ass}
	Let $\Omega$ be a unit ball in $\R^2$, that is, $\Omega := \{(x_1, x_2) \in \R^2 \mid x_1^2 + x_2^2 <1 \}$. Then $\bar y := \1_{\{ x_1^2 + x_2^2 < \frac{1}{2} \}}\left( \frac{1}{2} - x_1^2 - x_2^2 \right) \in H^1_0(\Omega) \cap C(\overline\Omega)$ satisfies the structural assumption \eqref{eq:structure-non-diff}. 
\end{example}

\begin{example}
	\label{exam:SA-GA}
	We consider  the optimal control problem
	\begin{equation*}
		\left\{
		\begin{aligned}
			\min_{u\in L^2(\Omega)} &j(u) := \frac{1}{2} \norm{y_u}_{L^2(\Omega)}^2 + \frac{\nu}{2} \norm{u}_{L^2(\Omega)}^2 \\
			\text{s.t.} \quad &-\Delta y_u + \max(0,y_u) = u \, \text{ in } \Omega, \quad y_u = 0 \, \text{ on } \partial\Omega.
		\end{aligned}
		\right.
	\end{equation*}
	Obviously, $\bar u =0$ is a unique minimizer and the corresponding state and adjoint state are, respectively,
	$\bar y =0$ and $\bar p =0$. Easily, \eqref{ass:G-diff} and \eqref{eq:structure-non-diff} are fulfilled and one has
	\[
		\text{meas}(\{ \bar y = 0\}) = \text{meas}(\Omega) >0.
	\]
\end{example}

\begin{lemma}
	\label{lem:Q-defined}
	Assume that \eqref{ass:G-diff} and \eqref{eq:structure-non-diff} are fulfilled. Then, for any $h \in L^2(\Omega)$, there holds    
	\begin{equation*}
		\left| \tilde{Q}(\bar u,\bar p;h) \right| \leq c_s \norm{\bar p}_{L^\infty(\Omega)} \norm{S'(\bar u;h)}_{L^\infty(\Omega)}^2.
	\end{equation*}    
\end{lemma}
\begin{proof}
	We now show for any $\{t_n\} \in c_0^+$ that
	\begin{equation}
		\label{eq:Q-bound}
		\left| \underline Q(\bar u,\bar p;\{t_n\},\{h\}) \right| \leq c_s \norm{\bar p}_{L^\infty(\Omega)} \norm{S'(\bar u;h)}_{L^\infty(\Omega)}^2,
	\end{equation}
	which, together with the definition of $\tilde{Q}$, yields the desired conclusion. To this target, setting $y_n := S(\bar u + t_n h)$, we have $y_n \to \bar y$ in $H^1_0(\Omega) \cap C(\overline\Omega)$. Put $\epsilon_n := \norm{y_n - \bar y}_{C(\overline\Omega)}$ and define the measurable functions
	\begin{equation*}
		T_n :=  y_n \left( \1_{\{\bar y > 0, y_n < 0 \}}  - \1_{\{ y_n > 0, \bar y < 0 \}} \right).
	\end{equation*}
	Obviously, there holds
	\begin{align*}
		0 \geq T_n \geq   (y_n-\bar y) \left( \1_{\{\bar y > 0, y_n < 0 \}}  - \1_{\{ y_n > 0, \bar y < 0 \}} \right)
	\end{align*}
	for a.e. in $\Omega$.	
	A simple computing gives
	\begin{align*}
		&\{\bar y > 0, y_n < 0 \} \subset \{ 0 < \bar y < \bar y - y_n \} \subset \{ 0 < \bar y < \epsilon_n \} \\
		\intertext{and}
		& \{\bar y < 0, y_n > 0 \} \subset \{ \bar y - y_n < \bar y < 0 \} \subset \{ -\epsilon_n < \bar y < 0 \}.
	\end{align*}
	We therefore derive
	\begin{equation}
		\label{eq:inclusion-key1}
		\{\bar y > 0, y_n < 0 \} \cup \{\bar y < 0, y_n > 0 \} \subset \{ 0 < |\bar y| < \epsilon_n \}
	\end{equation}
	and
	\begin{equation*}
		0 \geq T_n \geq - |y_n - \bar y| \1_{ \{ 0 < |\bar y| < \epsilon_n \} }. 
	\end{equation*}
	Since $\epsilon_n \to 0^+$, there exists an integer $n_0 \in \N$ such that $\epsilon_n \in (0,1)$ for all $n \geq n_0$. The structural assumption \eqref{eq:structure-non-diff} then yields
	\begin{equation}
		\label{eq:T-n-esti}
		\norm{T_n}_{L^1(\Omega)} \leq c_s \epsilon_n^2 \quad \text{for all } n \geq n_0.
	\end{equation}
	Besides, the assumption \eqref{ass:G-diff} implies that
	\begin{equation} \label{eq:identity-key}
		\begin{aligned}[b]
			\bar p \zeta(\bar u;t_n, h) & = \bar p y_n\left( \1_{\{\bar y\geq 0 \}}  - \1_{\{ y_n \geq 0 \}} \right)\\
			& =\bar p y_n\left( \1_{\{\bar y\geq 0, y_n < 0 \}}  - \1_{\{ y_n \geq 0, \bar y <0 \}} \right) \\
			& = \bar p y_n \left( \1_{\{\bar y > 0, y_n < 0 \}}  - \1_{\{ y_n > 0, \bar y <0 \}} \right)
		\end{aligned}
	\end{equation}
	for a.e. in $\Omega$ and thus that
	\begin{equation*}
		\int_\Omega \bar p \zeta(\bar u;t_n, h) \dx = \int_\Omega \bar p T_n \dx.
	\end{equation*}
	This, along with \eqref{eq:T-n-esti}, gives
	\begin{equation*}
		\left|\int_\Omega \bar p \zeta(\bar u;t_n, h) \dx \right| \leq c_s\norm{\bar p}_{L^\infty(\Omega)} \epsilon_n^2
	\end{equation*}
	for all $n \geq n_0$. We then derive \eqref{eq:Q-bound} from the definition of $\underline Q$ and \cref{prop:control-to-state}. 
\end{proof}

Importantly, $\tilde{Q}$ is positively homogeneous of degree $2$ in $h$ and so is $Q$.
\begin{lemma} \label{lem:homogeneity}
Under the assumptions \eqref{ass:G-diff} and \eqref{eq:structure-non-diff}, there hold
  \begin{align*}
  	\tilde{Q}(\bar u, \bar p;th) = t^2 \tilde{Q}(\bar u, \bar p;h) \\
    \intertext{and}
    Q(\bar u, \bar p;th)  = t^2 Q(\bar u, \bar p;h) 
  \end{align*}
  for all $h \in L^2(\Omega)$ and $t>0$.
\end{lemma}
\begin{proof}
	Taking $h \in L^2(\Omega)$, $\{t_n\} \in c_0^+$, and $t>0$ arbitrarily, we observe from \eqref{eq:zeta-func} that 
  	\begin{equation*}
    	\zeta(\bar u;t_n,th) = \zeta(\bar u;tt_n,h).
  	\end{equation*}
  	We then have from the definition of $\underline Q$ that
  	\begin{equation*}
    	\underline{Q}(\bar u, \bar p;\{t_n\}, \{th\}) = t^2 \underline Q(\bar u, \bar p;\{tt_n\}, \{h\}).
  	\end{equation*}
  	There thus holds that
	\begin{equation*}
	   \begin{aligned}[b]
	    \tilde{Q}(\bar u, \bar p; th) &= \inf \left\{ \underline{Q}(\bar u,\bar p; \{t_n\}, \{th\}) \mid \{t_n\} \in c_0^+ \right\}\\
	    & = \inf \left\{ t^2\underline{Q}(\bar u,\bar p; \{tt_n\}, \{h\}) \mid \{t_n\} \in c_0^+ \right\}\\
	    & = t^2 \inf \left\{ \underline{Q}(\bar u,\bar p; \{r_n\}, \{h\}) \mid \{r_n\} \in c_0^+ \right\} \quad (\text{by setting } r_n := t t_n)\\
	    & = t^2 \tilde{Q}(\bar u, \bar p; h),
	   \end{aligned}
	 \end{equation*}
	 which shows the positive homogeneity of degree $2$ in $h$ of $\tilde{Q}$. This and the positive homogeneity of $S'(\bar u; \cdot)$ imply the desired property of $Q$.
\end{proof}

\begin{lemma}
	\label{lem:invariant}
	Assume that \eqref{ass:G-diff} and \eqref{eq:structure-non-diff} are satisfied. Then, for any sequences $\{t_n\} \in c_0^+$ and $\{h_n\}, \{v_n\} \subset L^2(\Omega)$ such that $h_n \rightharpoonup h$ and $v_n \rightharpoonup h$ in $L^2(\Omega)$ for some $h \in L^2(\Omega)$, one has
	\begin{equation} \label{eq:weak-invariant}
		\lim_{n \to \infty} \frac{1}{t_n^2} \int_\Omega \bar p \left[ \zeta(\bar u; t_n, h_n) -\zeta(\bar u; t_n, v_n)   \right] \dx = 0
	\end{equation}
	and, in particular, 
	\begin{equation} \label{eq:wlsc-semi}
		\underline{Q}(\bar u, \bar p; \{t_n\},\{h_n\}) = \underline{Q}(\bar u, \bar p; \{t_n\},\{h\}) \geq \tilde{Q}(\bar u, \bar p; h).
	\end{equation}
\end{lemma}
\begin{proof}
	It suffices to prove \eqref{eq:weak-invariant}. To do this, we set $y_n := S(\bar u + t_n h_n)$, $z_n := S(\bar u + t_n v_n)$ and $\epsilon_n := \norm{y_n - \bar y}_{C(\overline\Omega)}$, $\kappa_n := \norm{z_n - \bar y}_{C(\overline\Omega)}$. Obviously, we see from \cref{prop:control-to-state} that 
	\begin{equation}
		\label{eq:limits-invar}
		\epsilon_n, \kappa_n \to 0^+ \quad \text{and} \quad \frac{\norm{y_n-z_n}_{C(\overline\Omega)}}{t_n} \to 0,
	\end{equation}
	as well as 
	\begin{equation} \label{eq:bounds}
		0 \leq \frac{\epsilon_n}{t_n}, \frac{\kappa_n}{t_n} \leq C
	\end{equation}
	for all $n\geq 1$ and for some constant $C$ independent of $n$.
	Analogous to \eqref{eq:identity-key}, we have
	\begin{align*}
		\bar p \zeta(\bar u; t_n, h_n) & = \bar p y_n \left( \1_{\{\bar y > 0, y_n < 0 \}}  - \1_{\{ y_n > 0, \bar y <0 \}} \right)\\
		\intertext{and}
		\bar p \zeta(\bar u; t_n, v_n) & = \bar p z_n \left( \1_{\{\bar y > 0, z_n < 0 \}}  - \1_{\{ z_n > 0, \bar y <0 \}} \right)
	\end{align*}
	for a.e. in $\Omega$ and for all $n \geq 1$. Subtracting these identities yields
	\begin{equation}
		\label{eq:iden-invariant}
		\bar p \left[\zeta(\bar u; t_n, h_n) -  \zeta(\bar u; t_n, v_n) \right]= \bar p (A_n + B_n)
	\end{equation} 
	for a.e. in $\Omega$ and for all $n \geq 1$. Here 
	\begin{equation*}
		A_n:=  (y_n - z_n)  \left( \1_{\{\bar y > 0, y_n < 0 \}}  - \1_{\{ y_n > 0, \bar y <0 \}} \right)
	\end{equation*}
	and
	\begin{equation*}
	 	B_n := z_n \1_{\{\bar y>0 \}}\left[ \1_{\{y_n < 0 \}} - \1_{\{z_n < 0 \}} \right] 
		- z_n \1_{\{\bar y <0 \}}\left[ \1_{\{y_n > 0 \}} - \1_{\{z_n > 0 \}} \right].
	\end{equation*}
	Similar to \eqref{eq:inclusion-key1}, there hold
	\begin{equation*}
		\{\bar y > 0, y_n < 0 \} \cup \{ y_n > 0, \bar y <0 \} \subset \{ 0< |\bar y| < \epsilon_n \}
	\end{equation*}
	and, consequently,
	\begin{equation*}
		|A_n| \leq |y_n - z_n| \1_{\{ 0< |\bar y| < \epsilon_n \}}.
	\end{equation*}	
	Applying the structural assumption \eqref{eq:structure-non-diff} yields
	\begin{equation} \label{eq:A-n}
		\norm{A_n}_{L^1(\Omega)} \leq c_s \norm{y_n  - z_n}_{C(\overline\Omega)} \epsilon_n
	\end{equation}
	for all $n$ large enough.
	On the other hand, $B_n$ can be rewritten in the form
	\begin{equation*} 		
		B_n = z_n \1_{\{\bar y>0 \}}\left[ \1_{\{y_n < 0, z_n \geq 0 \}} - \1_{\{z_n < 0, y_n \geq 0 \}} \right] 
		- z_n \1_{\{\bar y <0 \}}\left[ \1_{\{y_n > 0, z_n \leq 0 \}} - \1_{\{z_n > 0, y_n \leq 0 \}} \right].
	\end{equation*}
	For a.e. $x \in \{ \bar y >0  \} \cap \{y_n < 0, z_n \geq 0 \}$, we have
	\begin{equation*}
		0 \leq z_n(x) \leq |z_n(x) - y_n(x)| \quad \text{and} \quad 0 < \bar y(x) \leq \bar y(x) - y_n(x) \leq  \epsilon_n
	\end{equation*}
	and therefore
	\begin{equation*}
		 0 \leq z_n \1_{\{\bar y>0 \}}  \1_{\{y_n < 0, z_n \geq 0 \}} \leq |y_n-z_n| \1_{\{ 0 < \bar y \leq \epsilon_n \}}.
	\end{equation*}
	Similarly, there hold
	\begin{align*}
		& 0 \leq - z_n \1_{\{\bar y>0 \}}  \1_{\{y_n \geq 0, z_n < 0 \}} \leq |y_n-z_n| \1_{\{ 0< \bar y \leq \kappa_n \}},\\
		& 0 \leq  - z_n \1_{\{\bar y <0 \}}\1_{\{y_n > 0, z_n \leq 0 \}} \leq |y_n-z_n| \1_{\{ -\epsilon_n \leq \bar y <0 \}},\\
		\intertext{and}
		& 0 \leq z_n \1_{\{\bar y <0 \}}\1_{\{z_n > 0, y_n \leq 0 \}} \leq |y_n-z_n| \1_{\{ -\kappa_n \leq \bar y <0 \}}.
	\end{align*}
	These above estimates give
	\begin{equation*}
		|B_n| \leq |y_n - z_n| \1_{\{ 0< |\bar y| \leq \max(\epsilon_n, \kappa_n) \}},
	\end{equation*}
	which, together with \eqref{eq:structure-non-diff}, yields
	\begin{equation} \label{eq:B-n}
		\norm{B_n}_{L^1(\Omega)} \leq c_s \norm{y_n- z_n}_{C(\overline\Omega)} \max(\epsilon_n, \kappa_n)
	\end{equation}
	for all $n$ sufficiently large.
	Combing \eqref{eq:iden-invariant} with \eqref{eq:A-n} and \eqref{eq:B-n} gives
	\begin{equation*}
		\left| \int_\Omega \bar p \left[ \zeta(\bar u; t_n, h_n) -\zeta(\bar u; t_n, v_n)   \right] \dx \right|  \leq 2c_s \norm{\bar p}_{L^\infty(\Omega)} \norm{y_n- z_n}_{C(\overline\Omega)} \max(\epsilon_n, \kappa_n).
	\end{equation*}
	From this, \eqref{eq:limits-invar}, and \eqref{eq:bounds}, we derive \eqref{eq:weak-invariant}.
\end{proof}

As a result of \eqref{eq:wlsc-semi}, we provides here the vanishing of the term $\tilde{Q}$ for some special cases. 
In the following proposition, the continuity assumption of $\bar u$ can be attained under condition \eqref{eq:normal_OS} and some additional requirements on functions $\alpha$ and $\beta$, e.g., $(\alpha,\beta) =(-\infty, \infty)$ a.e. in $\Omega$ or $\alpha, \beta \in C(\overline\Omega)$. 
It is noted that we do not impose any condition on $\bar u$ on the boundary $\partial\Omega$. 

\begin{proposition}
	\label{prop:precise-form}
	Let \eqref{ass:G-diff} and \eqref{eq:structure-non-diff} be fulfilled. Then, for any $h \in L^2(\Omega)$, there holds
	\begin{equation}
		\label{eq:key-term-vanish}
		\tilde{Q}(\bar u,\bar p;h) = 0
	\end{equation}
	if $\bar u$ is continuous on $\Omega$ and if one of the following conditions holds
	\begin{enumerate}[label=(\roman*)]
		\item $\bar u =0$ a.e. in $\Omega$;
		\item $\bar u(x) >0$ for all $x \in \Omega$;
		\item $\bar u(x) <0$ for all $x \in \Omega$.
	\end{enumerate}
\end{proposition}
\begin{proof}
	\emph{(i)} Assume that $\bar u=0$ a.e. in $\Omega$. \eqref{ass:G-diff} then implies that $\bar p=0$ a.e. in $\Omega$. From this and the definition of $\tilde{Q}$, we have \eqref{eq:key-term-vanish}.
	
	\emph{(ii)} Assume that $\bar u(x) >0$ for all $x \in \Omega$ and $\bar u$ is continuous on $\Omega$. 	
	From the definition of $\tilde{Q}$, it suffices to show that 
	\begin{equation}
		\label{eq:vanish-Q}
		\underline{Q}(\bar u,\bar p;\{t_n\}, \{h\}) =0
	\end{equation}
	for any sequence $\{t_n\} \in c_0^+$. To this end, let $\{t_n\} \in c_0^+$ be arbitrary, but fixed.
	From this and the definition \eqref{eq:key-term-sn}, there exists a subsequence, denoted in the same way, of $\{n\}$ such that
	\begin{equation}
		\label{eq:vanish-Q2}
		\underline{Q}(\bar u,\bar p;\{t_n\}, \{h\}) = \lim\limits_{n\to \infty} \frac{1}{t_n^2} \int_\Omega \bar p \zeta(\bar u; t_n,h)dx.
	\end{equation}
	Due to the density of $C^\infty_0(\Omega)$ in $L^2(\Omega)$, a sequence $\{h_m\} \subset C^\infty_0(\Omega)$ exists and satisfies $h_m \to h$  in $L^2(\Omega)$ as $m\to \infty$. Fix $m$ and set
	\[
		\rho_m := \begin{cases}
			1 & \text{if } \max_{x \in K_m}|h_m(x)| = 0,\\
			\frac{\min_{x\in K_m}\bar u(x)}{\max_{x \in K_m}|h_m(x)|} & \text{if } \max_{x \in K_m}|h_m(x)| > 0
		\end{cases}
	\]
	with $K_m := \text{supp}(h_m)$. Since $K_m$ is a compact subset of $\Omega$ and $\bar u$ is continuous and positive on $\Omega$, one has $\rho_m >0$. 
	There thus holds that
	\begin{equation}
		\label{eq:almost-positive}
		\bar u(x) + t h_m(x) \geq 0 \quad \text{for all } t \in (0, \rho_m] \, \text{and } x \in \Omega.
	\end{equation}
	Since $t_n \to 0^+$ as $n \to \infty$, there exists a ${n_m} > n$ such that $t_{n_m} \in (0, \rho_m]$. We then derive
	\[
		\bar u(x) + t_{n_m}h_m(x) \geq 0 \quad \text{for all } x \in \Omega.
	\]
	Combining this and the maximum principle yields $S(u+t_{n_m}h_m) \geq 0$ a.e. in $\Omega$. On the other hand, since $\bar u >0$ a.e. in $\Omega$, there holds $\bar y \geq 0$ a.e. in $\Omega$. From this and the definition \eqref{eq:zeta-func}, we obtain $\zeta(\bar u; t_{n_m}, h_{m}) = 0$ and therefore
	\begin{equation} \label{eq:vanish3}
		\underline{Q}(\bar u,\bar p;\{t_{n_m}\}, \{h_{m}\}) = 0.
	\end{equation}
	Besides, \eqref{eq:vanish-Q2} gives
	\begin{align*}
		\underline{Q}(\bar u,\bar p;\{t_n\}, \{h\}) & = \lim\limits_{m\to \infty} \frac{1}{t_{n_m}^2} \int_\Omega \bar p \zeta(\bar u; t_{n_m},h)dx\\
		& = \underline{Q}(\bar u,\bar p;\{t_{n_m}\}, \{h\}).
	\end{align*}
	From this, \eqref{eq:wlsc-semi}, and the fact that $h_m \to h$ in $L^2(\Omega)$, we direve
	\[
		\underline{Q}(\bar u,\bar p;\{t_n\}, \{h\}) = \underline{Q}(\bar u,\bar p;\{t_{n_m}\}, \{h_m\}),
	\] 
	which, together with \eqref{eq:vanish3}, yields \eqref{eq:vanish-Q}.
	
	\emph{(iii)} Assume that $\bar u(x) <0$ for all $x \in \Omega$ and $\bar u$ is continuous on $\Omega$. Analogous to (ii), we also have \eqref{eq:key-term-vanish}.
\end{proof}

We now can employ \eqref{eq:wlsc-semi} to prove the weak lower semi-continuity in $h$ of $\tilde{Q}$. 
\begin{proposition}\label{prop:wlsc-key-term}
	If \eqref{ass:G-diff} and \eqref{eq:structure-non-diff} are satisfied, then 
	for any $h_n \rightharpoonup h$ in $L^2(\Omega)$,
	\begin{equation*}
		\tilde{Q}(\bar u,\bar p;h) \leq \liminf_{n \to \infty} \tilde{Q}(\bar u,\bar p;h_n).
	\end{equation*}
\end{proposition}
\begin{proof}
	The proof is similar to that in \cite[Prop.~5.6]{ClasonNhuRosch2020} with some slight modifications. For the convenience of the reader, we present here the detailed arguments.  
	Take $\{h_n\} \subset L^2(\Omega)$ arbitrarily such that $h_n \rightharpoonup h$ in $L^2(\Omega)$. 	
	Fixing $n \in\N$ and exploiting the definition \eqref{eq:key-term} yields that  sequences $\{t_j^k(h_n)\}_{j, k \in \N} \subset c_0^+$ exist and fulfill 
	\begin{equation} \label{eq:limit-wlsc}
		t_j^k(h_n) \to 0^+ \text{ as } j \to \infty \quad \text{for all } k \in\N
	\end{equation}
	and 
	\begin{equation*}
		\tilde{Q}(\bar u,\bar p;h_n) = \lim_{k \to \infty} \underline{Q}(\bar u,\bar p;\{t_j^k(h_n)\}_{j \in \N} , \{h_n\}_{j \in \N}).
	\end{equation*}
	This leads to the existence of a $k_n \geq n$ such that 
	\begin{equation}
		\label{eq:weakly-lsc-1}
		\tilde{Q}(\bar u,\bar p;h_n) - \underline{Q}(\bar u,\bar p;\{t_j^{k_n}(h_n)\}_{j \in \N} , \{h_n\}_{j \in \N}) > - \frac{1}{n}.
	\end{equation}
	From \eqref{eq:limit-wlsc}, there is a $j_n \in \N$ such that
	\begin{equation}
		\label{eq:lwsc-zero-sequence}
		0 < t_{j}^{k_n}(h_n) < \frac{1}{n} \quad \text{for all } j \geq j_n.
	\end{equation}
	Moreover, by virtue of \eqref{eq:key-term-sn}, a subsequence $\{t_{j_q}^{k_n}(h_n)\}_{q \in \N}$ of $\{t_j^{k_n}(h_n)\}_{j \in \N}$ exists and satisfies
	\begin{equation*}
		\underline{Q}(\bar u,\bar p;\{t_j^{k_n}(h_n)\}_{j \in \N} , \{h_n\}_{j \in \N}) = \lim_{q \to \infty} \frac{1}{(t_{j_q}^{k_n}(h_n))^2} \int_\Omega \bar p \zeta(\bar u; t_{j_q}^{k_n}(h_n), h_n)  \dx. 
	\end{equation*}	
	An integer $q_n \in \N$  therefore exists such that $j_{q_n} \geq j_n$ and
	\begin{equation}
		\label{eq:weakly-lsc-2}
		\underline{Q}(\bar u,\bar p;\{t_j^{k_n}(h_n)\}_{j \in \N} , \{h_n\}_{j \in \N}) - \frac{1}{r_n^2} \int_\Omega \bar p  \zeta(\bar u; r_n, h_n) \dx > - \frac{1}{n}
	\end{equation}
	with $r_n := t_{j_{q_n}}^{k_n}(h_n)$. Besides, we see from \eqref{eq:lwsc-zero-sequence} that $r_n \to 0^+$ as $n \to \infty$ and so $\{r_n\}_{n \in \N} \in c_0^+$. 	
	Furthermore, adding \eqref{eq:weakly-lsc-1} and \eqref{eq:weakly-lsc-2} yields that
	\begin{equation*}
		\tilde{Q}(\bar u,\bar p;h_n)  - \frac{1}{r_n^2} \int_\Omega \bar p  \zeta(\bar u; r_n, h_n) \dx  > -\frac{2}{n}.
	\end{equation*}	
	We now take the limit inferior to derive
	\begin{equation*} 
		\liminf _{n \to \infty} \tilde{Q}(\bar u,\bar p;h_n)  \geq \liminf _{n \to \infty}\frac{1}{r_n^2} \int_\Omega \bar p  \zeta(\bar u; r_n, h_n) \dx = \underline{Q}(\bar u, \bar p; \{r_n\}, \{h_n\})
	\end{equation*}
	which, in cooperation with \eqref{eq:wlsc-semi}, yields the claim.
\end{proof}

The weak lower semi-continuity of $\tilde{Q}$ leads to the same property of the non-smooth functional $Q$. 
\begin{corollary}
	\label{cor:wlsc-curvature}
	Under the assumptions of \cref{prop:wlsc-key-term}, $Q(\bar u, \bar p; \cdot)$ is weakly lower semi-continuous as a functional on $L^2(\Omega)$.
\end{corollary}
\begin{proof}
	The desired conclusion follows directly from \cref{prop:wlsc-key-term}, assertion (iv) in \cref{prop:control-to-state}, and the weak lower semi-continuity of the norm in $L^2(\Omega)$.
\end{proof}

\subsection{Second-order conditions}\label{sec:2nd-OC:soc}
 
We start this subsection with a second-order Taylor-type expansion.
\begin{lemma}
	\label{lem:2nd-expression}
	For any $u \in L^2(\Omega)$ and $y_u:= S(u)$, there holds for any $t>0$ that
	\begin{multline} \label{eq:2nd-Taylor-expansion}
		j(u) - j(\bar u) = \int_\Omega\int_0^1 (1-s)L''_{yy}(x,\bar y + s(y_u -\bar y))(y_u - \bar y)^2 \ds \dx \\
		\begin{aligned}[t]
			& + \frac{\nu}{2} \norm{u - \bar u}_{L^2(\Omega)}^2+ \int_\Omega \bar d(u - \bar u)\dx + \int_\Omega \bar p \zeta\left(\bar u; t, \frac{u-\bar u}{t}\right) \dx,
		\end{aligned}
	\end{multline}
	where $\bar p$ and $\bar d$ are defined in \eqref{eq:objective-der}.
\end{lemma}
\begin{proof}
	Exploiting a Taylor expansion yields
	\begin{multline} \label{eq:obj-diff}
		j(u) - j(\bar u)  =  \int_\Omega \left[L(x,y_u) - L(x, \bar y) \right]\dx + \frac{\nu}{2} \int_\Omega (u^2 - \bar u^2)\dx \\
		\begin{aligned}[b]				
				& = \int_\Omega L'_y(x, \bar y)(y_u - \bar y) \dx +\nu \int_\Omega  \left( u - \bar u \right)  \bar u\dx  + \frac{\nu}{2} \norm{u - \bar u}_{L^2(\Omega)}^2 \\
				\MoveEqLeft[-1] + \int_\Omega \int_0^1 (1-s)L''_{yy}(x,\bar y + s(y_u -\bar y))(y_u - \bar y)^2 \ds\dx  \\
				& = \int_\Omega L'_y(x, \bar y)(y_u - \bar y) \dx - \int_\Omega \bar p \left( u - \bar u \right)  \dx + \int_\Omega \bar d(u - \bar u)\dx  \\
				\MoveEqLeft[-1] + \int_\Omega \int_0^1 (1-s)L''_{yy}(x,\bar y + s(y_u -\bar y))(y_u - \bar y)^2 \ds\dx + \frac{\nu}{2} \norm{u - \bar u}_{L^2(\Omega)}^2 ,
	\end{aligned}
	\end{multline}
	where we have just used \eqref{eq:objective-der} to derive the last equality. 	
	Testing now the state equations for $y_u$ and $\bar y$ by $\bar p$ and then subtracting the obtained results gives
	\begin{equation*}
		\int_\Omega \bar p (u - \bar u) \dx = \int_\Omega \nabla (y_u -  \bar y)  \cdot \nabla \bar p + \bar p \left[\max(0,y_u) - \max(0, \bar y)\right] \dx.
	\end{equation*}
	Since $\bar p$ satisfies \eqref{eq:adjoint}, there holds
	\begin{equation*}
		\int_\Omega L'_y(x, \bar y)(y_u - \bar y) \dx = \int_\Omega \nabla (y_u -  \bar y)  \cdot \nabla \bar p +\1_{\{\bar y\geq 0\}} \bar p (y_u - \bar y) \dx.
	\end{equation*}
	It then follows that
	\begin{multline*}
		\int_\Omega L'_y(x, \bar y)(y_u - \bar y) \dx - \int_\Omega \bar p \left( u - \bar u \right)  \dx\\
		\begin{aligned}[t]
		 	&= \int_\Omega \bar p \left[ \1_{\{\bar y\geq 0\}}(y_u - \bar y) - \max(0,y_u) + \max(0, \bar y)\right]\dx \\
			& = \int_\Omega \bar p y_u \left[ \1_{\{\bar y\geq 0\}} - \1_{\{y_u \geq 0 \}} \right] \dx\\
			& = \int_\Omega \bar p \zeta\left(\bar u; t,\frac{u -\bar u}{t}\right) \dx,
		\end{aligned}
	\end{multline*}	
	where the last equality has been derived using \eqref{eq:zeta-func}. 
	Inserting this equality into \eqref{eq:obj-diff}, we arrive at the desired conclusion.
\end{proof}

We finally recall the following basic notion standard in the study of second-order conditions for optimal problem with control constraints; see, e.g.,\cite{BonnansShapiro2000,BayenBonnansSilva2014}.
Let $K$ be a closed subset in $L^2(\Omega)$ and let $z \in K$ be arbitrary. 
The symbols
\begin{align*}
    & \mathcal{R}(K; z) := \left\{ h \in L^2(\Omega)\,\middle|\, \exists \bar t > 0 \text{ s.t. } z + t h \in K\ \,\forall t \in [0,\bar t] \right\},\\
    & \mathcal{T}(K; z) := \left\{ h \in L^2(\Omega)\,\middle|\, \exists t_n \to  0^+, h_n \to h \text{ in } L^2(\Omega) \text{ s.t. } z + t_n h_n \in K \,\forall n \in\N  \right\},\\
    \intertext{and}
    & \mathcal{N}(K; z) := \left\{ w \in L^2(\Omega) \,\,\middle|\, \int_\Omega w(x) h(x)\dx \leq 0 \,\forall h \in \mathcal{T}(K; z) \right\}
\end{align*}
stand, respectively, for the \emph{radial}, \emph{contingent (Bouligand) tangent}, and \emph{normal} cones to $K$ at $z$.
If  $K$ is convex, then 
\begin{equation*}
    \mathcal{T}(K; z) = \text{cl}_2\left[ \mathcal{R}(K; z) \right],
\end{equation*}
where cl$_2(M)$ denotes the closure of a set $M$ in $L^2(\Omega)$; see, e.g., \cite{BonnansShapiro2000,AubinFrankowska2009}.
For any $w \in L^2(\Omega)$, by $w^\bot$ we denote the annihilator of $w$, i.e.,
\begin{equation*}
    w^\bot := \left\{ v \in L^2(\Omega) \,\middle|\, \int_\Omega w(x)v(x)\dx = 0 \right\}.
\end{equation*}
Moreover, the set $K$ is said to be \emph{polyhedric at $z \in K$} if for any $w \in \mathcal{N}(K;z)$, there holds
\begin{equation*}
    \text{cl}_2\left[ \mathcal{R}(K; z) \cap (w^\bot) \right] = \mathcal{T}(K; z) \cap (w^\bot).
\end{equation*}
We say the set $K$ \emph{polyhedric} if it is polyhedric at each point $z \in K$.

The following result stating the polyhedricity of the feasible set $\mathcal{U}_{ad}$ shall be employed to prove the second-order necessary optimality conditions for \eqref{eq:P}.
\begin{lemma}
	\label{lem:polyhedricity}
	The admissible set $\mathcal{U}_{ad}$ is polyhedric.
\end{lemma}
\begin{proof}
	For the case $\alpha = -\infty$ and $\beta = \infty$ a.e. in $\Omega$, we obviously have $\mathcal{U}_{ad} = L^2(\Omega)$ and thus $\mathcal{U}_{ad}$ is polyhedric. On the other hand, for the case $\alpha, \beta \in L^2(\Omega)$, the set $\mathcal{U}_{ad}$ is also polyhedric due to, for example,  \cite[Lem.~4.13]{BayenBonnansSilva2014}.
\end{proof}

We are now ready to prove the following two theorems, which are the main results of the paper and provide a no-gap theory of second-order necessary and sufficient conditions in terms of the non-smooth functional $Q$ defined in \cref{def:curvature-func}.

\begin{theorem}[second-order necessary optimality condition] 
    \label{thm:2nd-OS-nec}
    Assume that $\bar u$ is a local minimizer of \eqref{eq:P2} and that $j$ is G\^{a}teaux-differentiable in $\bar u$. Assume further that $\bar y:= S(\bar u)$ satisfies the structural assumption \eqref{eq:structure-non-diff}. Then the adjoint state $\bar p:= G_{\1_{\{ \bar y \geq 0 \}}}\left(L'_y(\cdot, \bar y) \right) \in H^1_0(\Omega) \cap C(\overline\Omega)$ fulfills \eqref{eq:1st-OS-relaxed} and \eqref{ass:G-diff} is satisfied. Moreover, the following second-order necessary optimality condition holds:
    \begin{equation}
        \label{eq:2nd-OS-nec}
        Q(\bar u, \bar p;h)  \geq 0 \qquad\text{for all }h\in \mathcal{C}(\mathcal{U}_{ad};\bar u).
    \end{equation}
    Here the \emph{critical} cone $\mathcal{C}(\mathcal{U}_{ad};\bar u)$ is defined via
    \begin{equation*}
	    \mathcal{C}(\mathcal{U}_{ad};\bar u) := \left\{ h \in L^2(\Omega) \,\middle|\, h(x) 
	    \begin{cases}
	    	\geq 0 & \text{if }  \bar u(x) = \alpha(x)\\
	    	\leq 0 & \text{if }  \bar u(x) = \beta(x) \\
	    	= 0 & \text{if }  \bar d(x) \neq 0
	    \end{cases} 
	    \text{a.e. } x \in \Omega 
	    \right\}
    \end{equation*}
    with $\bar d$ given in \eqref{eq:objective-der}.
\end{theorem}
\begin{proof}
	Since $j$ is G\^{a}teaux-differentiable at $\bar u$, $\bar y$ satisfies \eqref{ass:G-diff} due to \cref{thm:G-diff-characterization}. 
	As a result of \cref{thm:1st-OS-relaxed}, $\bar p:= G_{\1_{\{ \bar y \geq 0 \}}}\left(L'_y(\cdot, \bar y) \right) \in H^1_0(\Omega) \cap C(\overline\Omega)$ fulfills \eqref{eq:1st-OS-relaxed}. It remains to prove \eqref{eq:2nd-OS-nec}. For this purpose, let $h \in \mathcal{C}(\mathcal{U}_{ad};\bar u)$ and $\{t_n\} \in c_0^+$ be arbitrary but fixed. It suffices to show that
    \begin{equation}
        \label{eq:2nd-OS-nec-2}
        \frac{1}{2} \int_\Omega\left[ L''_{yy}(x,\bar y) \left(S'(\bar u; h)\right)^2 + \nu h^2 \right] \dx +  \underline{Q}(\bar u,\bar p;\{t_n\},\{h\}) \geq 0.
    \end{equation}    
    In order to get \eqref{eq:2nd-OS-nec-2}, we first have from \eqref{eq:key-term-sn} that a subsequence $\{t_{n_k}\}$ of $\{t_n\}$ exists and satisfies
    \begin{equation}
        \label{eq:sigma-subsequence}
        \underline{Q}(\bar u,\bar p;\{t_n\},\{h\}) =  \lim_{k \to \infty} \frac{1}{t_{n_k}^2} \int_\Omega \bar p \zeta(\bar u; t_{n_k}, h)  \dx.
    \end{equation}
    Thanks to \cite[Lem.~4.11]{BayenBonnansSilva2014}, the critical cone can be expressed via 
    \begin{equation}
    	\label{eq:critical-cone}
    	\mathcal{C}(\mathcal{U}_{ad};\bar u) = \mathcal{T}(\mathcal{U}_{ad}; \bar u) \cap (\bar d^\bot).
    \end{equation}    
    The polyhedricity of $\mathcal{U}_{ad}$; see \cref{lem:polyhedricity}, implies the existence of sequences $\{h_m\} \subset L^2(\Omega)$, $\{q_m\} \subset \mathcal{U}_{ad}$, and $\{\lambda_m \} \subset (0,\infty)$ satisfying
    \begin{equation*}
        h_m \to h \quad \text{in } L^2(\Omega), \quad h_m = \frac{q_m - \bar u}{\lambda_m}, \quad \text{and} \quad h_m \in \bar d^\bot \quad \text{for all } m \in\N.
    \end{equation*}
    Since $t_{n_k} \to 0^+$ as $k \to \infty$, a subsequence, denoted by $\{r_m\}$, of $\{t_{n_k}\}$ exists such that $ 0 < r_m \leq \lambda_m$ for all $m \in\N$. From this and the convexity of $\mathcal{U}_{ad}$, we have
    \begin{equation*}
        \bar u + r_m h_m = \left(1 - \frac{r_m}{\lambda_m}  \right)\bar u + \frac{r_m}{\lambda_m} q_m \in \mathcal{U}_{ad} \quad \text{for all } m \in\N,
    \end{equation*}
    which gives 
    \begin{equation*}
        \frac{1}{r_m^2}\left( j(\bar u + r_m h_m) - j(\bar u) \right) \geq 0
    \end{equation*}
    for all $m\in\N$ sufficiently large. We then derive from \cref{lem:2nd-expression} and the fact $h_m \in \bar d^\bot$ that  
   	\begin{multline*} 
   		\frac{1}{r_m^2}\int_\Omega\int_0^1 (1-s)L''_{yy}(x,\bar y + s(y_m -\bar y))(y_m - \bar y)^2 \ds \dx \\
   		\begin{aligned}[t]
   		& + \frac{\nu}{2} \norm{h_m}_{L^2(\Omega)}^2+\frac{1}{r_m^2} \int_\Omega \bar p \zeta\left(\bar u; r_m, h_m\right) \dx \geq 0
   		\end{aligned}
   	\end{multline*}   
   	with $y_m := S(\bar u + r_m h_m)$.     
    Taking  the limit inferior and exploiting the fact that  $h_m \to h$ in $L^2(\Omega)$, we can conclude from \cref{prop:control-to-state}, the Lebesgue dominated convergence theorem, and the definition of $\underline Q$ that
    \begin{equation*}
    	\frac{1}{2} \int_\Omega L''_{yy}(x, \bar y)[S'(\bar u;h)]^2 \dx + \frac{\nu}{2} \norm{h}_{L^2(\Omega)}^2 + 
         \underline{Q}(\bar u,\bar p; \{r_m\},\{h_m\}) \geq 0.
    \end{equation*}
    Besides, we have $\underline{Q}(\bar u,\bar p; \{r_m\},\{h_m\}) = \underline{Q}(\bar u,\bar p; \{r_m\},\{h\})$ by virtue of \cref{lem:invariant}. Moreover, since  $\{r_m\}$ is a subsequence of $\{t_{n_k}\}$, we deduce 
    \begin{equation*}
	    \underline{Q}(\bar u,\bar p; \{r_m\},\{h\}) = \underline{Q}(\bar u,\bar p; \{t_n\},\{h\})
    \end{equation*}
    from \eqref{eq:sigma-subsequence}. We then derive $\underline{Q}(\bar u,\bar p; \{r_m\},\{h_m\}) = \underline{Q}(\bar u,\bar p; \{t_n\},\{h\})$, which finally yields \eqref{eq:2nd-OS-nec-2}.
\end{proof}

\begin{theorem}[second-order sufficient optimality conditions] 
    \label{thm:2nd-OS-suf}
    Assume that $\bar u$ is an admissible point of \eqref{eq:P2} such that $\bar y:= S(\bar u)$ satisfies the structural assumption \eqref{eq:structure-non-diff}. Assume further that there exists an adjoint state $\bar p \in H^1_0(\Omega) \cap C(\overline\Omega)$ that together with $\bar u, \bar y$ fulfills \eqref{eq:1st-OS-relaxed} and \eqref{ass:G-diff} as well as
    \begin{equation}
        \label{eq:2nd-OS-suff}
         Q(\bar u,\bar p;h) >0
        \qquad\text{for all }h\in \mathcal{C}(\mathcal{U}_{ad};\bar u)\setminus \{0\}.
    \end{equation}
    Then constants $c, \rho >0$ exist and satisfy
    \begin{equation}
        \label{eq:quadratic-grownth}
        j(\bar u) + c \norm{u - \bar u}_{L^2(\Omega)}^2 \leq j(u) \quad \text{for all } u \in \mathcal{U}_{ad} \cap \overline B_{L^2(\Omega)}(\bar u, \rho).
    \end{equation}
    In particular, $\bar u$ is a strict local minimizer of \eqref{eq:P2}.
\end{theorem}
\begin{proof}
	We first observe from \eqref{ass:G-diff} and \eqref{eq:p-identity} that $\bar p = G_{\1_{\{\bar y \geq 0 \}}}\left(L'_y(\cdot, \bar y) \right)$. We now prove the theorem by a contradiction argument. Suppose the claim was false. Then there exists a sequence $\{u_n\} \subset \mathcal{U}_{ad}$ that fulfills
    \begin{equation} \label{eq:contradiction}
        \norm{u_n - \bar u}_{L^2(\Omega)} < \frac{1}{n} \quad \text{and} \quad j(\bar u) + \frac{1}{n}\norm{u_n - \bar u}_{L^2(\Omega)}^2 > j(u_n),\quad n\in\N.
    \end{equation}
    We put $t_n := \norm{u_n - \bar u}_{L^2(\Omega)}$ and $h_n := \frac{u_n - \bar u}{t_n}$ and thus have $\norm{h_n}_{L^2(\Omega)} = 1$ and $t_n \in c_0^+$. A subsequence of $\{h_n\}$, also denoted in the same way, exists and satisfies $h_n \rightharpoonup h$ in $L^2(\Omega)$ for some $h \in L^2(\Omega)$.
	Since $\mathcal{U}_{ad}$ is convex, one has $\mathcal{T}(\mathcal{U}_{ad};\bar u) = \text{cl}_2(\mathcal{R}(\mathcal{U}_{ad};\bar u))$ and therefore $\mathcal{T}(\mathcal{U}_{ad};\bar u)$ is weakly closed in $L^2(\Omega)$. Since $h_n \in \mathcal{R}(\mathcal{U}_{ad}; \bar u) \subset   \mathcal{T}(\mathcal{U}_{ad};\bar u)$ and $h_n \rightharpoonup h$ in $L^2(\Omega)$, $h \in \mathcal{T}(\mathcal{U}_{ad};\bar u)$. On the other hand, we deduce from \eqref{ass:G-diff} and \cref{thm:G-diff-characterization} that $j$ is G\^{a}teaux-differentiable in $\bar u$. 
    This, together with the last inequality in \eqref{eq:contradiction}, implies, for $n$ large enough, that
    \begin{equation*}
        j'(\bar u)(u_n - \bar u) + o(t_n) < \frac{1}{n} t_n^2.
    \end{equation*}
    Dividing the above estimate by $t_n$ and then passing to the limit yields $j'(\bar u)h \leq 0$.    
    \eqref{eq:diff-formula} and the definition of $\bar d$ then give 
    \begin{equation*}
	    \int_\Omega \bar d h \dx \leq 0.
    \end{equation*}
    This, in cooperation with \eqref{eq:normal_OS-relaxed}, implies that $\int_\Omega \bar d h \dx =0$. Hence, there holds that $h \in \bar d^\bot \cap \mathcal{T}(\mathcal{U}_{ad};\bar u)$ and so $h \in \mathcal{C}(\mathcal{U}_{ad};\bar u)$ according to \eqref{eq:critical-cone}.

    We now derive a contradiction and thus complete the proof. In fact, from the last inequality in \eqref{eq:contradiction}, we have
    \begin{equation*}
        \frac{1}{t_n^2} \left[j(u_n) - j(\bar u) \right] < \frac{1}{n}.
    \end{equation*}
    From this, \eqref{eq:normal_OS-relaxed} and \cref{lem:2nd-expression}, there holds
    \begin{multline*} 
    	\frac{1}{n} > \frac{1}{t_n^2}\int_\Omega\int_0^1 (1-s)L''_{yy}(x,\bar y + s(y_n -\bar y))(y_n - \bar y)^2 \ds \dx \\
    	\begin{aligned}[t]
    		& + \frac{\nu}{2} \norm{h_n}_{L^2(\Omega)}^2+\frac{1}{t_n^2} \int_\Omega \bar p \zeta\left(\bar u; t_n, h_n\right) \dx
    	\end{aligned}
    \end{multline*} 
    with $y_n := S(u_n)$.    
    Taking the limit inferior, exploiting \cref{prop:control-to-state}, and using that $h_n\weakto h$ in $L^2(\Omega)$, we arrive at
    \begin{equation*}
    	0 \geq \frac{1}{2} \int_\Omega L''_{yy}(x, \bar y)[S'(\bar u;h)]^2 \dx + \frac{\nu}{2} \norm{h}_{L^2(\Omega)}^2 + 
    	\underline{Q}(\bar u,\bar p; \{t_n\},\{h_n\}) +  \frac{\nu}{2}(1 -\norm{h}_{L^2(\Omega)}^2).
    \end{equation*}
    Exploiting \eqref{eq:wlsc-semi} therefore gives
    \begin{equation*}
    	0 \geq \frac{1}{2} \int_\Omega L''_{yy}(x, \bar y)[S'(\bar u;h)]^2 \dx + \frac{\nu}{2} \norm{h}_{L^2(\Omega)}^2 + 
    	\tilde{Q}(\bar u,\bar p; h) +  \frac{\nu}{2}(1 -\norm{h}_{L^2(\Omega)}^2),
    \end{equation*}
    which is identical to
    \begin{equation}  \label{eq:contradiction2}
        0 \geq \frac{1}{2} Q(\bar u, \bar p;h)+  \frac{\nu}{2}(1 -\norm{h}_{L^2(\Omega)}^2).
    \end{equation}
    Since $\norm{h_n}_{L^2(\Omega)} = 1$ and the norm in $L^2(\Omega)$ is weakly lower semicontinuous, there holds $\norm{h}_{L^2(\Omega)} \leq 1$. 
    From this, \eqref{eq:contradiction2}, and \eqref{eq:2nd-OS-suff}, we have $h = 0$. Inserting $h=0$ into \eqref{eq:contradiction2} yields 
    \begin{equation*}
	    0 \geq \frac{\nu}{2} >0,
    \end{equation*}
    which is impossible.
\end{proof}

We finish this section by stating an equivalent version of the sufficient second-order optimality condition \eqref{eq:2nd-OS-suff} that might be useful for estimating discretization errors in finite element approximations. Its proof is similar to that of \cite[Prop.~5.11]{ClasonNhuRosch2020} and is thus omitted.
\begin{proposition}
    \label{prop:2nd-OS-suf2}
    Assume that $\bar u$ is an admissible point of \eqref{eq:P2} such that $\bar y:=S(\bar u)$ fulfills the structural assumption \eqref{eq:structure-non-diff}. Assume further that  there exists an adjoint state $\bar\varphi \in H^1_0(\Omega) \cap C(\overline\Omega)$ that as well as $\bar u, \bar y$ satisfies \eqref{eq:1st-OS-relaxed} and \eqref{ass:G-diff}. Then, \eqref{eq:2nd-OS-suff} holds if and only if there exist constants $c_0, \tau > 0$ such that
    \begin{equation}
        \label{eq:2nd-OS-suff-equiv}
        Q(\bar u,\bar p;h) \geq  c_0 \norm{h}_{L^2(\Omega)}^2 \qquad\text{for all } h \in \mathcal{C}^\tau(\mathcal{U}_{ad};\bar u)
    \end{equation}
    for
    \begin{equation*}
        \mathcal{C}^\tau(\mathcal{U}_{ad};\bar u) := \left\{ h \in L^2(\Omega) \,\middle|\, h(x) 
            \begin{cases}
                \geq 0 & \text{if }  \bar u(x) = \alpha(x)\\
                \leq 0 & \text{if }  \bar u(x) = \beta(x) \\
                = 0 & \text{if }  |\bar d(x)| > \tau 
            \end{cases} 
            \text{a.e. } x \in \Omega
        \right\}.
    \end{equation*}
\end{proposition}

\section{Conclusions}

We have derived second-order optimality conditions for an optimal control problem governed by a semilinear elliptic differential equation with the non-smooth max-function nonlinearity.
The G\^{a}teaux-differentiability of the objective functional in any control variable is precisely characterized by the vanishing of an adjoint state on the  set of all points at which the values of the corresponding state coincide with the non-differentiability point of the max-function. 
Under the G\^{a}teaux-differentiability of the objective functional in the considered control and an structural assumption, we define second-order generalized derivatives of the objective functional. A second-order Taylor-type expansion thus formulates the necessary and sufficient optimality conditions. A no-gap theory of second-order optimality conditions then follows. Finally, an equivalent formulation of the second-order sufficient optimality condition that might be employed for discretization error estimates is also obtained. Such estimates will be studied in a future work.

\bibliographystyle{tfnlm}
\bibliography{nogapsemilinear}

\end{document}